\documentclass[11pt]{article}
\usepackage{geometry}
\usepackage[utf8]{inputenc}
\usepackage[english]{babel}

\usepackage{authblk}

\usepackage{amsmath,amssymb,amsthm}
\usepackage{physics} 
\usepackage[justification=centering]{caption}
\usepackage[colorlinks=true, citecolor=red, linkcolor=blue]{hyperref}
\usepackage[nameinlink, capitalize]{cleveref}

\usepackage[colorinlistoftodos]{todonotes}
\usepackage{varwidth}

\usepackage{color}
\usepackage{float}
\usepackage[super]{nth}
\usepackage{booktabs}
\usepackage{rotating}
\usepackage{verbatim}
\usepackage{multirow}
\usepackage{physics}
\usepackage{mathtools}
\usepackage{subcaption}
\usepackage{bm}
\usepackage[toc,page]{appendix}
\usepackage[shortlabels]{enumitem}
\usepackage{nomencl}

\theoremstyle{plain}
\newtheorem{theorem}{Theorem}
\newtheorem{corollary}{Corollary}

\newtheorem{proposition}{Proposition}

\theoremstyle{definition}

\newtheorem{definition}{Definition}
\newtheorem{remark}{Remark}

\numberwithin{equation}{section}

\title{Velocity Averaging for the Wigner Equation with Electromagnetic Fields}
\author[a]{François Golse}
\author[a,b,c]{Jakob Möller}
\affil[a]{CMLS, École polytechnique, F-91128 Palaiseau}
\affil[b]{Research Platform MMM "Mathematics-Magnetism-Materials" c/o Fak. Math., Univ. Wien, A-1090 Vienna}
\affil[c]{Wolfgang Pauli Institute, A-1090 Vienna}

\begin{document}
\maketitle

\begin{abstract}
    We discuss the application of $L^2$-based velocity averaging to the Wigner kinetic equation governing the quantum evolution of the Wigner transform of a density operator of a particle subjected to an external electromagnetic field. This extends the $L^2$-based averaging lemma from [F.Golse, J.Möller: Commun. Math. Sci. (2027)] to electromagnetic fields. We make use of the magnetic or gauge-invariant Wigner transform due to Stratonovich [Sov. Phys.
Dokl., 1:414–418, 1956], which differs from the usual definition by a phase factor involving the circulation of the magnetic vector potential along a line segment. The magnetic Wigner transform obeys the Wigner equation in the kinetic phase space variables which is a more natural setting for the application of velocity averaging. Since the magnetic Wigner equation contains a first order derivative in the spatial variable, we obtain compactness in $L^2_{\mathrm{loc}}$ uniformly in the semiclassical parameter but no gain in regularity. 
\end{abstract}

\setcounter{tocdepth}{1}

\section{Introduction}

The classical kinetic evolution of a particle with phase space density $f\colon \mathbb{R}_t\times\mathbb{R}^3_x\times \mathbb{R}^3_v \rightarrow \mathbb{R}_+$,  accelerated by an electromagnetic field $E=-\nabla \phi -\partial_t A$, $B=\nabla \times A \in \mathbb{R}^3$ where $A\colon \mathbb{R}_t\times \mathbb{R}^3_x \rightarrow \mathbb{R}^3$ is the magnetic vector potential, while $\phi\colon \mathbb{R}_t\times \mathbb{R}^3_x \rightarrow \mathbb{R}$ is the electric scalar potential, is given by the electromagnetic transport equation
\begin{equation}
    \partial_t f + v\cdot \nabla_x f  + \nabla_v \cdot((E + v \times B )f) = 0.
    \label{eq: magnetic Vlasov}
\end{equation}
Note that vector differential operators such as $\nabla, \Delta$ are taken with respect to $x$, sometimes with respect to $x$ and $v$ and never with respect to $t$. When not mentioned, $\nabla$ means $\nabla_x$ and $\Delta$ means $\Delta_x$. Moreover, while we work in $d=3$ whenever the magnetic field $B$, the cross product $\times$ or the Levi-Civita symbol $\varepsilon_{ijk}$ are involved, we keep a general dimension $d$ whenever the object makes sense in any dimension.

The quantum version of the electromagnetic transport equation is given by the \emph{magnetic Wigner equation} for the \emph{magnetic Wigner transform} $W^{\hbar}_A$
\begin{equation}
    i\hbar \partial_t W_A^{\hbar} = \{\{H,W_A^{\hbar}\}\}_B -i\hbar\nabla_v\cdot(\Theta[E]W_A^{\hbar}), \quad H(t,x,v) = \frac{1}{2}|v|^2.
    \label{eq: magnetic wigner moyal equation}
\end{equation}
where $\{\{\cdot,\cdot\}\}_B$ denotes the \emph{magnetic Moyal bracket} and where $\Theta[\cdot]$ is a pseudo-differential operator acting on the $v$-variable, defined in \eqref{eq:def theta}. The magnetic Moyal bracket is defined via the magnetic Weyl quantization, cf. \eqref{eq:KernelMagneticWeyl},
\begin{equation}
   \mathrm{Op}^W_A[\{\{f,g\}\}_B] := [\mathrm{Op}^W_A[f],\mathrm{Op}^W_A[g]],
\end{equation}
where we omit the $t$-dependence of $A=A(t)$ for simplicity. The magnetic Moyal bracket is the quantum deformation of the \emph{magnetic Poisson bracket} $$\{f,g\}_B := \sum_{k=1}^3 (\partial_{x_k}f\partial_{v_k}g-\partial_{v_k}f\partial_{x_k}g) + \sum_{i,j,m}^3 \epsilon_{ijm}B_m \partial_{v_i}f\partial_{v_j}g.$$
The explicit form of the magnetic Wigner equation \eqref{eq: magnetic wigner moyal equation} is given in Proposition \ref{thm:magnetic equations}. The \emph{magnetic Wigner transform} $W_A^{\hbar}$ \cite{mantoiuMagneticWeylCalculus2004, mullerProductRuleGauge1999, serimaaGaugeindependentWignerFunctions1986, stratonovichGaugeInvariantAnalog1956} is defined as 
\begin{equation}
   W^{\hbar}_A(t,x,v):=\frac{1}{(2\pi)^d}\int_{\mathbb R^d}R(t,x+\tfrac\hbar{2}y,x-\tfrac\hbar{2}y)e^{\frac{i}\hbar\Gamma_{A(t)}[x+\frac\hbar{2}y,x-\frac\hbar{2}y]}e^{-iv\cdot y}d y,
    \label{eq: Wigner magnetic}
\end{equation}
where $R$ is the density operator of the quantum state, cf. Section \ref{sec:preliminaries}, and where $\Gamma_A[X,Y]$ denotes the circulation of $A$ along the line segment joining $X$ and $Y$,
\begin{equation*}
        \Gamma_A[X,Y]:=\int_0^1A((1-s)X+sY)\cdot (Y-X) d s.
    \end{equation*}
The density operator $R$ satisfies the von Neumann equation
\begin{align}
\label{eq:VonNeumann}
\begin{cases}
    i\hbar \partial_t R(t) = [\hat{H}(t),R(t)], \quad \hat{H}(t) = \frac{1}{2}(-i\hbar\nabla-A(t))^2 +\phi(t), \\
    R(0) = R^{\mathrm{in}}.
\end{cases}
\end{align}
Under suitable conditions on the potentials $(A,\phi)$ the Hamilton operator $\hat{H}$ is self-adjoint and generates a unitary propagator $U(t,s)$ such that the solution $R(t)$ is given by
\begin{equation}
\label{eq:UnitaryEvolution}
    R(t) = U(t,0)R^{\mathrm{in}}U(t,0)^{\ast}.
\end{equation}
The magnetic Wigner equation can be written in the following form
\begin{align}
\label{eq:wigner caricature}
\begin{cases}
    \partial_t W_A^{\hbar} + v\cdot \nabla_x W_A^{\hbar} 
    = \Phi^{\hbar}[E^{\hbar},B^{\hbar}]W_A^{\hbar}, \\
    W_{A}^{\hbar,\mathrm{in}} = W_{A(0)}^{\hbar}[R^{\mathrm{in}}],
\end{cases}
\end{align}
where $\Phi^{\hbar}[E^{\hbar},B^{\hbar}]$, defined in \eqref{eq:Phi}, is an integro-differential operator acting on $W^{\hbar}_A$ that involves the electromagnetic field $(E^{\hbar},B^{\hbar}) = (-\nabla \phi^{\hbar}-\partial_t A^{\hbar},\nabla \times A^{\hbar})$ which might depend on $\hbar$, cf. Proposition \ref{thm:magnetic equations}. Note that the explicit form of the magnetic Wigner equation was recently derived in \cite{nedjalkovWignerEquationGeneral2019}. 

Compare this to the ``minimally coupled" Wigner equation
\begin{equation}
    i\hbar \partial_t W^{\hbar} = \{\{H,W^{\hbar}\}\}, \quad H(t,x,\xi) =\frac{1}{2}|\xi-A(t,x)|^2 + \phi(t,x),
    \label{eq: wigner moyal equation}
\end{equation}
where $W^{\hbar}$ is the \emph{non-magnetic} Wigner transform of the density matrix $R$:
\begin{equation}
    W^{\hbar}(t,x,\xi) = \frac{1}{(2\pi)^d} \int_{\mathbb{R}^d} e^{-i \xi \cdot y} R(t,x+\tfrac{\hbar y}{2},x-\tfrac{\hbar y}{2}) d y.
    \label{eq: Wigner non magnetic}
\end{equation}
Here, $\{\{\cdot,\cdot\}\}$ denotes the Moyal bracket defined via the non-magnetic Weyl quantization
\begin{equation}
    \mathrm{Op}^W[\{\{f,g\}\}] := [\mathrm{Op}^W[f], \mathrm{Op}^W[g]]
\end{equation}
and it is the quantum deformation of the canonical Poisson bracket
\begin{equation*}
    \{f,g\} := \sum_{k=1}^d \partial_{x_k}f\partial_{\xi_k}g-\partial_{\xi_k}f\partial_{x_k} g.
\end{equation*}
The term \emph{minimal coupling} refers to the replacement of the canonical momentum $\xi$ by the kinetic momentum 
\begin{equation}
    v \equiv v(t,x,\xi) := \xi-A(t,x).
\end{equation}
For a more detailed discussion of canonical and magnetic Poisson brackets and their quantization we refer to Sections \ref{sec:PoissonBrackets} and  \ref{sec:MagneticQuantization}.

If $W_A$ is the Wigner measure of $W^{\hbar}_A$, 
then – for $E$ and $B$ sufficiently regular – the right-hand side $\Phi^{\hbar}[E,B]W^{\hbar}_A$ of \eqref{eq:wigner caricature} should converge to the forcing term $-\nabla_v \cdot ((E+ v \times B) W_A)$ of the electromagnetic transport equation \eqref{eq: magnetic Vlasov} in the classical limit $\hbar \rightarrow 0$. Therefore, if the non-magnetic Wigner transform $W^{\hbar}(t,x,\xi)$ converges to a Wigner measure $W(t,x,\xi)$ solving the minimally coupled Liouville equation \eqref{eq: Liouville}, then the magnetic Wigner transform $W^{\hbar}_A(t,x,v)$ converges to $W_A(t,x,v) = W(t,x,v+A(t,x)) = W \circ T^{-1}(t,x,v)$ where $T(t,x,\xi) := (x,\xi-A(t,x))$ denotes the change of variables from the canonical momentum variable $\xi$ to the kinetic momentum variable $v = \xi-A(t,x)$. \\ 

The main goal of this paper is to establish a semiclassical averaging lemma (i.e. an averaging lemma that holds uniformly in the semiclassical regime $\hbar \rightarrow 0$, cf. \cite{golseVelocityAveragingNot2027, golseVelocityAveragingLemmas2026}) for the magnetic Wigner equation \eqref{eq:wigner caricature}. More precisely, we show in Theorem \ref{thm:main} below that the family of averages in velocity, indexed by $\hbar$,
\begin{equation*}
    \rho^{\hbar}_{\varphi}(W^{\hbar}_A)(t,x) := \int_{\mathbb{R}^3} W^{\hbar}_A(t,x,v) \varphi(v) d v,
\end{equation*}
where $\varphi$ is some compactly supported, sufficiently regular cutoff function, is relatively compact as a subset of $L^2_{\mathrm{loc}}(\mathbb{R}_t\times \mathbb{R}^3_x)$.

\begin{theorem}
\label{thm:main}
    Let $E^{\hbar} ,B^{\hbar} \in L^{\infty}(\mathbb{R}_t \times \mathbb{R}_x^3)$ uniformly in $\hbar$ be families of electromagnetic fields
    indexed by $\hbar$ such that  
    \begin{align*}
            \sup_{0<\hbar\leq 1}\|\hbar (\nabla \times B^{\hbar})\|_{L^{\infty}} \leq C.
    \end{align*}
    Let $R^{\hbar} = \sum_{j\geq1} \lambda_j^{\hbar}|\psi^{\hbar}_j\rangle \langle{\psi^{\hbar}_j}|$ be a family of density operators satisfying \begin{equation}
        \label{eq:condition mixed states}\tr_{L^2(\mathbb{R}^3)}((R^{\hbar})^2) = \|R^{\hbar}\|_{L^2(\mathbb{R}^3\times\mathbb{R}^3)}^2 = \sum_{j=1}^{\infty} (\lambda_j^{\hbar})^2 \leq C(2\pi\hbar)^3,
    \end{equation} and which are given by \eqref{eq:UnitaryEvolution}. Let $W_A^{\hbar} = W^{\hbar}_A[R^{\hbar}]$ be their magnetic Wigner transform \eqref{eq: Wigner magnetic}, obeying the magnetic Wigner equation \eqref{eq:wigner caricature}. Let $\varphi \in C^2_c(\mathbb{R}^3)$. Then the family of averages
    \begin{equation*}
        \rho^{\hbar}_{\varphi}(W^{\hbar}_A)(t,x) := \int_{\mathbb{R}^3} W^{\hbar}_A(t,x,v) \varphi(v) d v,
    \end{equation*}
    satisfies the following bound for the finite differences $\delta_z\rho^{\hbar}_{\varphi}(x) := \rho^{\hbar}_{\varphi}(x+z)-\rho^{\hbar}_{\varphi}(x)$,
    \begin{equation}
    \label{eq:modulus of continuity Wigner}
            \|\delta_z \rho^{\hbar}_{\varphi}(W^{\hbar}_A)\|_{L^2} \leq C \left(\hbar^{\frac{1}{6}}+|z|^{\frac{1}{6}}\right)\|W^{\hbar}_A\|_{L^2},
        \end{equation}
        where the constant $C$ depends on $\varphi$, $\sup_{\hbar}\|E^{\hbar}\|_{L^{\infty}}$, $\sup_{\hbar}\|B^{\hbar}\|_{L^{\infty}}$ and $\sup_{\hbar}\hbar\|\nabla \times B^{\hbar}\|_{L^{\infty}}$.
        
    In particular, as $\hbar \rightarrow 0$, the family of averages $\rho^{\hbar}_{\varphi}(W^{\hbar}_A)$ is relatively compact in $L^2([-T,T]\times B(0,R))$ for all $T>0$ and for all $R>0$.
\end{theorem}

We prove Theorem \ref{thm:main} in Section \ref{sec:VA} by an application of a variant of the theorem of Perthame-Souganidis in \cite{perthameLimitingCaseVelocity1998}, cf. Corollary \ref{corollary AL}.

\begin{remark}
    The main advantage of equation \eqref{eq:wigner caricature} is its similarity to the electromagnetic transport equation \eqref{eq: magnetic Vlasov}, from which it differs by a term of order $O(\hbar)$. As long as this term is controlled uniformly in $\hbar$, we obtain compactness of the averages in $L^2_{\mathrm{loc}}$. However, due to a full spatial derivative in $\Phi^{\hbar}[E,B]$ we do not obtain any gain in regularity uniformly in $\hbar$, as opposed to the non-magnetic case \cite{golseVelocityAveragingNot2027, golseVelocityAveragingLemmas2026}, where one gains $1/4$ of a derivative uniformly in $\hbar$.

Moreover, it is not clear how to establish an averaging lemma for the ``minimally coupled" Wigner equation \eqref{eq: wigner moyal equation}, as it contains terms of the form
\begin{equation}
    (\xi-A(t,x))\cdot \nabla_x W^{\hbar}.
    \label{eq:TransportTermWigner}
\end{equation}
When averaging in the canonical momentum variable $\xi$, one propagates $W^{\hbar}$ along the flow of the operator 
\begin{equation*}
    \partial_t + \xi \cdot \nabla_x,
\end{equation*}
and views $A(t,x)\cdot \nabla_x W^{\hbar}$ as a member of the right-hand side. In this case however, we cannot apply the Perthame-Souganidis theorem as this term is not $O(\hbar)$, in fact it is even $O(1)$, except when $A \equiv A^{\hbar}$ and $A^{\hbar} \rightarrow 0$ as $\hbar \rightarrow 0$, a case which is not very interesting. This is the main reason to use the magnetic Wigner transform. The magnetic Wigner equation has the advantage that it contains the change of variables $v=\xi-A$ already at the quantum level and \eqref{eq:TransportTermWigner} is replaced by
\begin{equation}
    v\cdot \nabla_x W^{\hbar}_{A},
\end{equation}
while all the other terms containing a full derivative in $x$ are $O(\hbar)$, and no full spatial derivative terms that are $O(1)$ (and thus would prevent compactness) such as \eqref{eq:TransportTermWigner} are present. Compare also Remark \ref{RemarkFullDerivative}.
\end{remark}

\begin{remark}
In view of the unitary evolution of $R(t)$, cf. \eqref{eq:UnitaryEvolution}, the Hilbert-Schmidt norm of $R(t)$ is conserved by the flow of the magnetic von Neumann equation. Therefore it suffices to assume that \eqref{eq:condition mixed states} holds initially.
   \end{remark}

\begin{remark}
\label{RemarkPurevMixed}
    In Theorem \ref{thm:main} we show an averaging lemma for a special class of \emph{mixed states}, excluding pure states via condition \eqref{eq:condition mixed states}; compare the discussion in Section \ref{sec:preliminaries} and Proposition \ref{thm:condition mixed states}, which uses condition \eqref{eq:condition mixed states} in order to derive a uniform $L^2$ bound for $W^{\hbar}_A$. The (im-)possibility of an application of velocity averaging to the Wigner equation corresponding to the non-magnetic Schrödinger equation with external potential was discussed in \cite{golseVelocityAveragingNot2027}, where velocity averaging for a certain class of pure states converging to monokinetic Wigner measures with zeroth- and first-order moments converging strongly in $L^2$ was ruled out. 
    
    In \cite{golseVelocityAveragingNot2027}, we proved an $L^2$-based averaging lemma based on the averaging lemma in \cite{dipernaGlobalWeakSolutions1989}, which however does not allow for a full derivative in $x$ on the right-hand side of the transport equation. This is overcome by an application of Corollary \ref{corollary AL}, which is a consequence of Theorem 1 in \cite{perthameLimitingCaseVelocity1998} and which includes a full derivative in $x$.
\end{remark}

\begin{remark}
    The exponent $\frac{1}{6}$ on the right-hand side of \eqref{eq:modulus of continuity Wigner} stems from the fact that $\Phi^{\hbar}[E^{\hbar},B^{\hbar}]$, defined in \eqref{eq:Phi}, contains a second order derivative in $v$ (albeit of order $\hbar$), such that the usual exponent in the $L^2$-based averaging lemma becomes $$\frac{1}{2(m+1)} =\frac{1}{6},$$ with $m=2$.

    Moreover, the term scaling with $\hbar^{\frac{1}{6}}$ comes from the term in $\Phi^{\hbar}[E^{\hbar},B^{\hbar}]$  containing the derivative in $x$, which disappears as $\hbar\rightarrow 0$. The term scaling with $|z|^{\frac{1}{6}}$ comes from the terms in $\Phi^{\hbar}[E^{\hbar},B^{\hbar}]$ containing no derivative in $x$ (but possibly derivatives in $v$) and would otherwise generate a gain in differentiability for $\rho_{\varphi}$ uniformly in $\hbar$. This is obstructed by the full spatial derivative in $\Phi^{\hbar}[E^{\hbar},B^{\hbar}]$.
\end{remark}

\begin{remark}
\label{RemarkFullDerivative}
    The requirement that $\|\hbar (\nabla \times B^{\hbar})\|_{L^{\infty}}$ is bounded uniformly in $\hbar$ is needed since $\Phi^{\hbar}$ contains a term of the form $\hbar T[\nabla \times B^{\hbar}]W^{\hbar}_A$, where $T[f]$ is a pseudo-differential operator depending on a function $f$ and acting on $W^{\hbar}_A$. In fact, this term can be removed, by combining it with the term containing the full derivative in $x$. Compare the first term in \eqref{eq:Phi} which can be written as
    \begin{equation*}
        \hbar T[B^{\hbar}]\times \nabla_x W^{\hbar}_A := \hbar \epsilon_{ijk}T[B^{\hbar}_i]\partial_{x_j} W^{\hbar}_A,
    \end{equation*}
    However, in this case, we cannot apply Corollary \ref{corollary AL}, resp. Theorem \ref{thm:ALPerthameSouganidis}, since it is not clear how to control a term of the form $T\partial_x f$ instead of $\partial_x (Tf)$ if the operator $T$ depends on $x$.

    In the context of Maxwell's equations, $\nabla \times B^{\hbar}$ is equal to the sum of the current density $J^{\hbar}$ and the displacement current $\partial_t E^{\hbar}$. The condition basically says that the current $J^{\hbar}+\partial_t E^{\hbar}$ should not blow up with rate $\frac{1}{\hbar}$ in the semiclassical regime.
\end{remark}

\subsection{Canonical versus magnetic Poisson brackets}
\label{sec:PoissonBrackets}
The electromagnetic transport equation \eqref{eq: magnetic Vlasov} can be derived from the Liouville equation
\begin{equation}
    \partial_t f = \{H,f\},
    \label{eq: Liouville}
\end{equation}
where $H$ is the minimally coupled Hamiltonian 
\begin{equation}
    H(t,x,\xi) = \frac{1}{2}|\xi-A(t,x)|^2 + \phi(t,x).
    \label{eq:magnetic Hamiltonian}
\end{equation}
Here, $\{\cdot,\cdot\}$ denotes the Poisson bracket, defined as
\begin{equation*}
    \{f,g\} := \sum_{k=1}^d \partial_{x_k}f\partial_{\xi_k}g-\partial_{\xi_k}f\partial_{x_k} g,
\end{equation*}
which is induced by the symplectic form
\begin{equation*}
    \omega := \sum_{k=1}^d d x_k \wedge d \xi_k.
\end{equation*}
Minimal coupling refers to the replacement of the canonical momentum $\xi$ by the kinetic momentum $v  =v(t,x,\xi) := \xi-A(t,x)$.

The transport equation \eqref{eq: magnetic Vlasov} is then obtained from \eqref{eq: Liouville} by a change of variables $v=\xi-A$. 
One notices that the Poisson brackets between the kinetic momenta $\xi_k-A_k$ and $\xi_{\ell}-A_{\ell}$ do not vanish, instead they are given by the magnetic field:
\begin{align*}
    \{x_k,x_\ell\} = 0, && \{x_k,\xi_{\ell}-A_{\ell}\} = \delta_{k\ell}, && \{\xi_k-A_k,\xi_{\ell}-A_{\ell}\} = \epsilon_{k\ell m} B_m,
\end{align*}
where $\epsilon_{k\ell m}$ denotes the Levi-Civita symbol. Alternatively, one can define the \emph{magnetic Poisson bracket} which takes the deformation by the kinetic momentum $v=\xi-A$ into account:
\begin{equation}
\label{eq:MagneticPoissonBracket}
    \{f,g\}_B := (\sum_{k=1}^3 \partial_{x_k}f\partial_{v_k}g-\partial_{v_k}f\partial_{x_k}g) + \sum_{i,j,m}^3 \epsilon_{ijm}B_m \partial_{v_i}f\partial_{v_j}g.
\end{equation}
More precisely, $\{f,g\}_B$ is induced by the magnetic symplectic form
\begin{equation*}
    \omega_B := \sum_{k=1}^3 d x_k \wedge d v_k - \frac{1}{2}\sum_{i,j,m=1}^3 \epsilon_{ijm}B_m d x_i \wedge d x_j.
\end{equation*}
By noting that
\begin{equation*}
    \{\xi_k-A_k,\xi_{\ell}-A_{\ell}\} = \{v_k,v_{\ell}\}_B,
\end{equation*}
we obtain that \eqref{eq: Liouville} is equivalent to 
the magnetic Liouville equation
\begin{equation}
    \partial_t f = \{ H',f\}_B - E \cdot \nabla_v f, \quad H'= \frac{1}{2} |v|^2, \quad E= -\nabla_x \phi - \partial_t A.
    \label{eq:magnetic Liouville}
\end{equation}
Then \eqref{eq:magnetic Liouville} is equivalent to the electromagnetic transport equation \eqref{eq: magnetic Vlasov} without having to perform the change of variables $v=\xi-A$. Thus, in classical theory, the formulation involving the minimally coupled Hamiltonian $H=\tfrac{1}{2}|\xi-A(x)|^2+\phi(x)$ on the usual phase space $(\mathbb{R}^3_x\times \mathbb{R}^3_{\xi},\omega)$ is equivalent to the formulation involving the Hamiltonian $H'=\tfrac{1}{2}|v|^2$ with the electric field term $-E\cdot \nabla_v f$ on the ``magnetic" phase space $(\mathbb{R}^3_x\times \mathbb{R}^3_{v},\omega_B)$.

\subsection{Quantization of magnetic Hamilton functions}
\label{sec:MagneticQuantization}
The equivalence of the canonical and magnetic Poisson brackets in classical mechanics is in general not conserved by the usual Weyl quantization of the Hamilton function $H$. The Weyl quantization of a Weyl symbol $a\colon \mathbb{R}^d_x \times \mathbb{R}^d_{\xi} \rightarrow \mathbb{R}$ is defined as
\begin{equation*}
    (\mathrm{Op}^W[a] f)(x) = \frac{1}{(2\pi\hbar)^d}\int e^{i\frac{(x-y)\cdot \xi}{\hbar}}a\left(\tfrac{x+y}{2},\xi\right) f(y)d\xi dy.
\end{equation*} 
The Weyl quantization $\mathrm{Op}^W$ of a minimally coupled symbol $a(x,\xi)$, where $\xi$ is replaced by $\xi-A(x)$, is not gauge covariant. This was already observed by Stratonovich \cite{stratonovichGaugeInvariantAnalog1956}.
For example, it is known that the relativistic magnetic Schrödinger operator $$\hat{H}_{\mathrm{rel}} := \sqrt{1+(-i\hbar\nabla-A)^2} +\phi,$$ when defined as the Weyl operator $\mathrm{Op}^W[P]$ with symbol
\begin{equation*}
    P(x,\xi) = \sqrt{1+|\xi-A(x)|^2} +\phi(x),
\end{equation*}
is \emph{not} gauge covariant \cite[Proposition 2.8]{ichinose_magnetic_2013}. Here, gauge covariance of an operator $H_A$ depending on the magnetic potential $A$ means that for any real-valued $C^{1}$ function $\chi$ it holds that
\begin{equation*}
    H_{A+\nabla \chi} = e^{i\chi/\hbar}H_A e^{-i\chi/\hbar}.
\end{equation*}

Thus, instead of a minimal coupling of the Weyl symbol, one has to modify the quantization itself. The usual Weyl quantization of the Hamiltonian \eqref{eq:magnetic Hamiltonian} is given by the position and momentum operators
\begin{align}
\label{eq:canonical operators}
Q_k\psi(x) :=x_k\psi(x) && V_k\psi(x):=\left(-i\hbar\partial_{x_k}-A_k(x)\right)\psi(x),
\end{align}
which obey the commutation relations
\begin{equation*}
[Q_k,Q_\ell]=0\,,\quad[Q_k,V_\ell]=i\hbar\delta_{k\ell}\,,\quad[V_k,V_\ell]=i\hbar(\partial_kA_\ell-\partial_\ell A_k)=i\hbar B_j\epsilon_{jk\ell}.
\end{equation*}
Then the quantization of the Hamiltonian \eqref{eq:magnetic Hamiltonian} is given by
\begin{equation*}
    \hat{H} = \frac{1}{2}V^2+\phi = \frac{1}{2}(-i\hbar\nabla-A)^2 + \phi.
\end{equation*}
If a  density operator $R$ is propagated by the von Neumann equation \eqref{eq:VonNeumann}
then the non-magnetic Wigner transform $W^{\hbar}[R]$ of $R$, cf. \eqref{eq: Wigner non magnetic}, satisfies the non-magnetic Wigner equation \eqref{eq: wigner moyal equation}
which is the quantum analogue of the minimally coupled Liouville equation \eqref{eq: Liouville}. As mentioned above, \eqref{eq: wigner moyal equation} is in general not covariant under a gauge transformation $A\rightarrow A+ \nabla \chi$. Note however that in the important subclass of quadratic Hamiltonians the non-magnetic Wigner equation \eqref{eq: wigner moyal equation} is indeed gauge-covariant. In fact, this holds for all Hamiltonians that are polynomials of $\xi$ of order \emph{at most} 2, cf. \cite[pp. 29-30]{lein_2011_thesis}. \\

The \emph{gauge-invariant} or \emph{magnetic Weyl quantization} \cite{lein_2011_thesis,mantoiuMagneticWeylCalculus2004,mullerProductRuleGauge1999,  stratonovichGaugeInvariantAnalog1956} is now defined as the integral operator $\mathrm{Op}^W_A[a]$ with symbol $a$ and kernel $K_A[a]$
\begin{equation}
\label{eq:KernelMagneticWeyl}
K_A[a](X,Y):=\frac{1}{{(2\pi\hbar)^d}}e^{-\frac{i}\hbar\Gamma_A[X,Y]}\int_{\mathbb R^d}a(\tfrac{X+Y}{2}, v)e^{i v\cdot \frac{X-Y}{\hbar}}{d v}.
\end{equation}
Therefore
\begin{align*}
\frac{1}{(2\pi)^d}\int_{\mathbb R^d}&K_A[a](x+\tfrac\hbar{2}y,x-\tfrac\hbar{2}y)e^{\frac{i}\hbar\Gamma_A[x+\frac\hbar{2}y,x-\frac\hbar{2}y]}e^{-iv\cdot y}d y
\\
&=\frac{1}{{(2\pi\hbar)^d}}\int_{\mathbb R^{d}}{a(x, w)}\left(\frac{1}{(2\pi)^d}\int_{\mathbb R^d}e^{i(w-v)\cdot y}d y\right)dw
\\
&=\frac{a(x,v)}{(2\pi\hbar)^d}.
\end{align*}
The magnetic Wigner transform \eqref{eq: Wigner magnetic} is defined in such a way that
\begin{equation}
W^{\hbar}_A[\mathrm{Op}^W_A[a]]=\frac1{(2\pi\hbar)^d}a.
\end{equation}
The magnetic Weyl quantization is now gauge-covariant, i.e.
\begin{equation*}
    \mathrm{Op}^W_{A+\nabla \chi}[a] = e^{i\chi/\hbar}\mathrm{Op}^W_{A}[a] e^{-i\chi/\hbar}.
\end{equation*}
for $\chi \in C^1$ real-valued.

\subsection{Organization of the article}

In Section \ref{sec:preliminaries} we define density operators, Weyl variables and discuss the mixed state condition \eqref{eq:condition mixed states}. In Section \ref{sec:magnetic wigner} we define the magnetic Wigner transform (or magnetic Wigner function) and prove that (under suitable conditions on the magnetic potential $A$ and the density operator) the magnetic Wigner function is uniformly bounded in $L^2$ and converges to a Wigner measure. In Section \ref{sec:Magnetic von Neumann and Wigner equations}, we derive the magnetic von Neumann and Wigner equations and introduce the respective notation, in particular we define the operator $\Phi^{\hbar}$. In Section \ref{sec:VA}, we prove Theorem \ref{thm:main}. We prove a variant (Corollary \ref{corollary AL}) of the velocity averaging lemma due to Perthame-Souganidis \cite{perthameLimitingCaseVelocity1998}, which includes a full derivative in $x$. We then use Corollary \ref{corollary AL} to finalize the proof of Theorem \ref{thm:main}.

\section{Preliminaries and notation}
\label{sec:preliminaries}
Let $\mathfrak{H}$ be a Hilbert space. The algebra of bounded operators on $\mathfrak{H}$ is denoted by $\mathcal{L}(\mathfrak{H)}$. We denote the two-sided ideal of trace-class operators on $\mathfrak{H}$ 
by $\mathcal{L}^1(\mathfrak{H})$. A \emph{density matrix} $R\in \mathcal{L}^1(\mathfrak{H})$ is a self-adjoint ($R=R^*$), nonnegative ($\langle R\psi,\psi\rangle \geq 0$ for all $\psi \in \mathfrak{H}$) operator 
with trace $\tr_{\mathfrak{H}} R = 1$. A rank-one density operator is a \emph{pure state}.  If $\mathfrak{H} = L^2(\mathbb{R}^d)$, the operator $R$ is an integral operator 
with integral kernel $(X,Y)\mapsto R(X,Y)$ which belongs to $L^2(\mathbb{R}^d\times \mathbb{R}^d)$. For the sake of simplicity, we shall abusively denote by $R(X,Y)$ the integral kernel of the density operator
$R$ --- so that $R$ designates both an operator on $L^2(\mathbb{R}^d)$ and an element of $L^2(\mathbb{R}^d\times \mathbb{R}^d)$. Since $R=R^*$ is trace-class, it is compact and by the spectral theorem 
there is a sequence of real eigenvalues $(\lambda_j)_{j\ge 1}$ and a Hilbert basis (i.e. a complete orthonormal system) $\{\psi_j\,:\,j\ge 1\} \subset \mathfrak{H}$ such that
\begin{equation}\label{SpecDecR}
    R(X,Y) = \sum_{j=1}^{\infty} \lambda_j \psi_j(X) \overline{\psi_j(Y)}.
\end{equation}
Since $R$ is a density operator,
\begin{equation}\label{CondLambda}
\lambda_j\ge 0\quad\text{since }R=R^*\ge 0,\quad\text{ and }\sum_{j\ge 1}\lambda_j=\tr_{L^2(\mathbb R^d)}R=1.
\end{equation}
In particular, the restriction of the integral kernel to the diagonal, 
i.e. $R(X,X)$, is a well-defined element of $L^1(\mathbb{R}^d_X)$ and we can define the density function $\rho\in L^1(\mathbb{R}^d)$ as 
\begin{equation}\label{DefDensityFunc}
    \rho(X) := R(X,X).
\end{equation}
We recall that the trace of $R$ and $R^2$ are given by
\begin{align*}
    \tr_{L^2(\mathbb R^d)}(R) &= \int_{\mathbb{R}^d} \rho(X) d X = \sum_{j\geq 1}\lambda_j, \\ \tr_{L^2(\mathbb R^d)}(R^2) &= \iint_{\mathbb{R}^d\times \mathbb{R}^d} |R(X,Y)|^2 d X dY= \sum_{j\geq 1}\lambda_j^2.
\end{align*}
Note that condition \eqref{eq:condition mixed states} essentially rules out density matrices corresponding to pure states, i.e. rank-one projections. Indeed,
by the Cauchy-Schwarz inequality,
\begin{equation*}
    1 = \left( \sum_{j\geq 1}\lambda_j\right)^2 \leq \sum_{j\geq 1} \mathrm{1}_{\lambda_j>0} \sum_{j\geq 1}\lambda_j^2 \leq C(2\pi\hbar)^d\cdot\mathrm{rank}(R),
\end{equation*}
so that
\begin{equation*}
    \mathrm{rank}(R) \geq \frac{1}{C(2\pi\hbar)^d}.
\end{equation*}
Thus pure states (corresponding to $\mathrm{rank}(R)=1$) are ruled out for $L^2$-based velocity averaging in the semiclassical regime, cf. Remark \ref{RemarkPurevMixed}.

The Weyl variables $(x,y)$ are defined as
\begin{align}
x:=\frac12(X+Y) && y:=\frac1\hbar(X-Y)
\label{eq:weyl variables}
\end{align}
so that
\[
X=x+\frac\hbar{2}y\,,\qquad Y=x-\frac\hbar{2}y.
\]
The terminology \emph{Weyl variables} is not standard and was used in \cite{filbetApproximationVonNeumannEquation2025}.
We denote all quantities in Weyl variables $(x,y)$ by a tilde, i.e.
\begin{align*}
    R(X,Y) = \tilde{R}(x,y), && \Gamma_A[X,Y] = \widetilde{\Gamma_A}[x,y], && S(X,Y) = \tilde{S}(x,y).
\end{align*}
For $d=3$ the completely antisymmetric Levi-Civita symbol is defined as
\begin{equation}
    \varepsilon_{ijk} = \begin{cases}
        1 \quad \text{if } (ijk) \text{ is an even permutation}, \\
        -1 \quad \text{if } (ijk) \text{ is an odd permutation}, \\
        0 \quad \text{otherwise}.
    \end{cases}
\end{equation}
The cross product $a \times b$ is then defined as
\begin{equation}
    (a \times b)_i = \sum_{j,k} \varepsilon_{ijk}a_jb_k.
\end{equation}

Furthermore, we denote by
\begin{equation}
    \delta_zf(x) := f(x+z)-f(x)
\end{equation}
the finite differences of a function $f$.

\section{Magnetic Wigner function}
\label{sec:magnetic wigner}
    Let $A(x)\in \mathbb{R}^d$ be a vector field. Recall the circulation of $A$ along the line segment $[X,Y]$ 
    \begin{equation}
        \label{eq:CirculationA}\Gamma_A[X,Y]:=\int_0^1A((1-s)X+sY)\cdot (Y-X) d s.
    \end{equation}
    The magnetic Wigner transform is defined as follows. 
\begin{definition}
    Let $K$ be an $L^2$-kernel. Let $\Gamma_A[X,Y]$ be the circulation of the vector field $A$ through the line segment $[X,Y]$, as in \eqref{eq:CirculationA}. The \emph{magnetic Wigner transform} is defined as 
    \begin{equation}
    \label{eq:magnetic wigner transform}
        W^{\hbar}_A[K](x,v):=\frac{1}{(2\pi)^d}\int_{\mathbb R^d}K(x+\tfrac\hbar{2}y,x-\tfrac\hbar{2}y)e^{\frac{i}\hbar\Gamma_A[x+\frac\hbar{2}y,x-\frac\hbar{2}y]}e^{-iv\cdot y}d y.
    \end{equation}
\end{definition}

   \begin{proposition}
   \label{thm:condition mixed states}
    Let $R^{\hbar}\in \mathcal{L}^1(L^2(\mathbb{R}^d))$ be a family of density operators indexed by $\hbar$ with Wigner transform $W^{\hbar}[R^{\hbar}](x,\xi)$ and corresponding Wigner measure $W(x,\xi)$. Suppose that $R^{\hbar}$ satisfies
\begin{equation}
    \|R^{\hbar}\|_{\mathrm{HS}}^2 = \sum_{j=1}^{\infty} (\lambda^{\hbar}_j)^2 \leq C (2\pi\hbar)^d.
    \label{eq:condition mixed states proposition}
\end{equation}Then
   \begin{enumerate}[label=(\roman*)]
\item For any measurable $A$, the magnetic Wigner transform satisfies the uniform bound
\begin{equation}
\|W_A^{\hbar}\|_{L^2_{x,v}}\leq C,
\label{eq:bound Wigner L^2}
\end{equation}
where $C$ is independent of $\hbar$. Thus, the family of magnetic Wigner transforms $W^{\hbar}_A$ is weakly relatively compact in $L^2_{x,v}$.
\item \label{ConvergenceWignerMagnetic} If $A\in \dot{H}^1(\mathbb{R}^d)$\footnote{Here, $\dot{H}^s$ denotes the homogeneous Sobolev space of order $s$, defined by the norm $$\|f\|^2_{\dot{H}^s} := \int |\xi|^{2s}|\widehat{f}(\xi)|^2 d\xi.$$}, then the magnetic Wigner transform $W^{\hbar}_A[R^{\hbar}](x,v)$ converges weakly in $L^2$ to the Wigner measure $W(x,v+A(x))$ along subsequences where $W^{\hbar}$ converges weakly to $W$.
\item If $A^{\hbar}$ is $\hbar$-dependent and converges strongly to $A$ in $L^2_{\mathrm{loc}}$ with uniformly bounded $\dot{H}^1$-norm, then the same conclusion holds as in \ref{ConvergenceWignerMagnetic}.
\end{enumerate}
   
\begin{proof}
For the first statement we have
\begin{equation}
     \|W^{\hbar}_A\|^2_{L^2_{x,v}} = \frac{1}{(2\pi)^d} \iint \left| e^{\frac{i}{\hbar}\widetilde{\Gamma}_A[x,y]} \tilde{R}(x,y)\right|^2 d x d y = \frac{1}{({2\pi})^d}\|\tilde{R}\|_{L^2_{x,y}}^2 =\frac{1}{(2\pi\hbar)^d}  \|R\|_{L^2_{X,Y}}^2 \leq C.
\end{equation}

Recall now the non-magnetic Wigner transform, cf. \eqref{eq: Wigner non magnetic}:
\begin{equation}
    W^{\hbar}(x,\xi) = \frac{1}{(2\pi)^d} \int_{\mathbb{R}^d} e^{-i \xi \cdot y} R(x+\tfrac{\hbar y}{2},x-\tfrac{\hbar y}{2}) d y.
\end{equation}
The non-magnetic Wigner transform $W_{\hbar}$ of a positive trace class operator $R$ converges weakly in $\mathcal{S}'(\mathbb{R}^d_x \times \mathbb{R}^d_{\xi})$ to a positive Radon measure $W$ on $\mathbb{R}^d_x\times \mathbb{R}^d_{\xi}$ (after extracting a subsequence), cf. \cite[Proposition III.1 and Theorem III.1]{lionsMesuresWigner1993}. By condition \eqref{eq:condition mixed states proposition} we have
\begin{equation}
     \|W^{\hbar}\|^2_{L^2_{x,\xi}} = \frac{1}{({2\pi})^d}\|\tilde{R}\|_{L^2_{x,y}}^2 =\frac{1}{(2\pi\hbar)^d}  \|R\|_{\mathrm{HS}}^2 \leq C.
\end{equation}
Thus $W^{\hbar} \rightharpoonup W$ in $L^2$ along subsequences. Define now 
\begin{equation}
    T(x,\xi) := (x,\xi-A(x))
\end{equation}
such that 
\begin{equation}
    W(x,v+A(x)) = (W\circ T^{-1})(x,v).
\end{equation}
The map $T$ is a bijection on $\mathbb{R}^d_x\times \mathbb{R}^d_{\xi}$ with Jacobian equal to one for any measurable function $A$, so in particular
\begin{equation*}
    \|\psi\|_{L^2_{x,v}} = \|\psi\circ T\|_{L^2_{x,\xi}}.
\end{equation*}
Now write
\begin{align*}
    \langle \varphi,W^{\hbar}_A\rangle_{x,v} = \langle \varphi \circ T,W^{\hbar}\rangle_{x,\xi} + \langle \varphi,W^{\hbar}_A-W^{\hbar}\circ T^{-1}\rangle_{x,v}
\end{align*}
for any $\varphi \in \mathcal{S}(\mathbb{R}^d_x \times \mathbb{R}^d_v)$.
The first term converges to $\langle \varphi \circ T,W\rangle$ by the weak convergence of $W^{\hbar}$. For the second we observe
\begin{align*}
    &\langle  \varphi(x,v),W^{\hbar}_A(x,v)-W^{\hbar}(x,v+A(x))\rangle \\ &\qquad \qquad = \tfrac{1}{(2\pi)^d} \iint_{\mathbb{R}^{2d}}\varphi(x,v) \left(\int_{\mathbb{R}^{d}} \left(e^{-iv\cdot y} e^{\frac{i}{\hbar}\widetilde{\Gamma}_A[x,y]}-e^{-i(v+A(x))\cdot y}\right)\tilde{R}_{\hbar}(x,y) dy\right) dx dv \\  &\qquad \qquad= \tfrac{1}{(2\pi)^d}\iint_{\mathbb{R}^{2d}} \widehat{\varphi}(x,y)e^{-iy\cdot A(x)}  \left(e^{\frac{i}{\hbar}\widetilde{\Gamma}_A[x,y]+iy \cdot A(x)}-1\right)\tilde{R}_{\hbar}(x,y) dxdy \\
    &\qquad \qquad= \tfrac{1}{(2\pi)^d}\iint_{\mathbb{R}^{2d}} \widehat{\varphi}(x,y)e^{-iy\cdot A(x)} \left( e^{-iy\cdot \int_{-1/2}^{1/2}(A(x+\hbar s y)-A(x))ds}-1\right)\tilde{R}_{\hbar}(x,y)dxdy.
\end{align*}
Taking absolute values and using that $\|\tilde{R}_{\hbar}\|_{L^2_{x,y}} \leq C$ (by condition \eqref{eq:condition mixed states proposition}) we obtain by Cauchy-Schwarz
\begin{align*}
    &\left|\langle  \varphi(x,v),W^{\hbar}_A(x,v)-W^{\hbar}(x,v+A(x))\rangle\right| \\ &\qquad \qquad \leq \|\tilde{R}_{\hbar}\|_{L^2_{x,y}}\left(\iint_{\mathbb{R}^{2d}} |\widehat{\varphi}(x,y)|^2 \left| e^{iy\cdot \int_{-1/2}^{1/2}A(x+\hbar s y)-A(x)ds}-1\right|^2 dx dy\right)^{\frac{1}{2}} \\
    &\qquad \qquad \leq C\left(\iint_{\mathbb{R}^{2d}} |\widehat{\varphi}(x,y)|^2 |y|^2\left|\int_{-1/2}^{1/2}A(x+\hbar s y)-A(x)ds\right|^2 dx dy\right)^{\frac{1}{2}} \\
    &\qquad \qquad \leq C\left(\int_{\mathbb{R}^{d}}\left(\sup_{x} |y|^2|\widehat{\varphi}(x,y)|^2\right) \left(\int_{\mathbb{R}^{d}}\left|\int_{-1/2}^{1/2}A(x+\hbar s y)-A(x)ds\right|^2 dx \right)dy \right)^{\frac{1}{2}}
    \\
    &\qquad \qquad \leq C\left(\int_{\mathbb{R}^{d}}\left(\sup_{x} |y|^2|\widehat{\varphi}(x,y)|^2\right) \left(\int_{-1/2}^{1/2}\left(\int_{\mathbb{R}^{d}}\left|A(x+\hbar s y)-A(x)\right|^2  dx\right)ds\right)dy \right)^{\frac{1}{2}}.
\end{align*}
Now since $A\in \dot{H}^{1}(\mathbb{R}^d)$ we have
\begin{align*}
    \int_{-1/2}^{1/2}\left(\int_{\mathbb{R}^{d}}\left|A(x+\hbar s y)-A(x)\right|^2  dx\right)ds &= \int_{-1/2}^{1/2}\left(\int_{\mathbb{R}^{d}}|\widehat{A}(\xi)|^2\left|e^{-i\hbar s y\cdot \xi}-1\right|^2  d\xi\right)ds\\ &\leq  \int_{-1/2}^{1/2}\left(\int_{\mathbb{R}^{d}}\left(\hbar |s||y||\xi|\right)^2|\widehat{A}(\xi)|^2 d\xi\right)ds\\&= \tfrac{1}{12}\hbar^2 |y|^2\|A\|^2_{\dot{H}^1}.
\end{align*}
Thus,
\begin{align*}
    \left|\langle  \varphi(x,v),W^{\hbar}_A(x,v)-W^{\hbar}(x,v+A(x))\rangle\right| &\leq C \hbar \|A\|_{\dot{H}^1}  \left(\int_{\mathbb{R}^{d}} |y|^4 \sup_x|\widehat{\varphi}(x,y)|^2 dy\right)^{\frac{1}{2}} \\
    &\leq C \hbar \|A\|_{\dot{H}^1} \|\varphi\|_{H^s_x\dot{H}^2_v}
\end{align*}
for $s>d/2$ by Sobolev's inequality. Therefore $W^{\hbar}_A(x,v) -W^{\hbar}(x,v+A(x))$ converges to zero in $\mathcal{S}'$ as $\hbar \rightarrow 0$. By a density argument and the boundedness of $W^{\hbar}_A$ in $L^2$ the second statement is proved.

For the last statement, we write
\begin{align*}
    \langle \varphi,W^{\hbar}_A\rangle_{x,v} = \langle \varphi \circ T,W^{\hbar}\rangle_{x,\xi} + \langle \varphi,W^{\hbar}_A-W^{\hbar}\circ (T^{\hbar})^{-1}\rangle_{x,v} + \langle \varphi\circ T^{\hbar} - \varphi \circ T, W^{\hbar}\rangle_{x,\xi}
\end{align*}
The first term on the right-hand side converges as before and the second term converges to zero by the uniform in $\hbar$ bound $\dot{H}^1$ of $A^{\hbar}$, as before. For the last term we observe that
\begin{align*}
    \|(\varphi\circ T^{\hbar} - \varphi \circ T)\|^2_{L^2_{x,\xi}} \leq \int \min\{4\|\varphi(x,\cdot)\|^2_{L^2_v},|A^{\hbar}(x)-A(x)|^2\|\nabla_v\varphi(x,\cdot)\|_{L^2_v}^2\} dx
\end{align*}
which converges to zero by dominated convergence. This concludes the proof.
\end{proof}
\end{proposition} 

\subsection{Magnetic von Neumann and Wigner equations}
\label{sec:Magnetic von Neumann and Wigner equations}
Let $R$ be a density matrix and define
\begin{equation}
    \label{eq:weyl-stratonovich}S(X,Y):=R(X,Y)e^{\frac{i}\hbar\Gamma_A[X,Y]}.
\end{equation}
Expression \eqref{eq:weyl-stratonovich} is also known as \emph{Weyl-Stratonovich transform}, cf. \cite{nedjalkovWignerEquationGeneral2019}.
In Weyl variables \eqref{eq:weyl variables} we have
\begin{equation}
\label{eq:weyl-stratonovich_tilde}
    \tilde{S}(x,y):=\tilde{R}(x,y)e^{\frac{i}\hbar\widetilde{\Gamma_A}[x,y]}.
\end{equation}
Recall that the magnetic Wigner transform \eqref{eq:magnetic wigner transform} is given by
\begin{equation*}
    W_A^{\hbar}[R](x,v)=\frac{1}{(2\pi)^d}\mathcal F_{y\to v}[\tilde S(x,\cdot)](v).
\end{equation*}
Let now $d=3$ and define the magnetic field as
\begin{equation}
    B_i := (\nabla \times A)_i := \varepsilon_{ijk}\partial_j A_k, \quad i=1,2,3.
\end{equation} 
We define the differential operators
\begin{align}
    \mathcal{D}_{1k} &:= i\partial_{y_k} - \hbar (y \times \mathcal{I}_1[B])_k, & \mathcal{I}_1[f] &:= \int_{-1/2}^{1/2} f(x+s\hbar y) s d s, \label{eq:DiffOpCurved1}\\
    \mathcal{D}_{2k} &:= \partial_{x_k} + i (y \times  \mathcal{I}_2[B])_k, & \mathcal{I}_2[f] &:= \int_{-1/2}^{1/2} f(x+\sigma \hbar y) d \sigma. \label{eq:DiffOpCurved2}
\end{align}
We also define the pseudo-differential operators $\Theta$ and $\Omega$,
\begin{align}
\begin{cases}
    \Theta[f]\psi := \tfrac{1}{(2\pi)^d}\mathcal{F}_{y\rightarrow v}\left[\mathcal{I}_2[f]\right]\ast_v \psi, \\
    \Omega[f]\psi := \tfrac{1}{(2\pi)^d}\mathcal{F}_{y\rightarrow v}\left[\mathcal{I}_1[f] \right]\ast_v \psi. 
    \end{cases}\label{eq:def theta}
\end{align}
We deduce the following magnetic von Neumann- and magnetic Wigner equations. For the proof of Proposition \ref{thm:magnetic equations} we refer to Appendix \ref{app:derivation magnetic equations}.
\begin{proposition}
\label{thm:magnetic equations} Let $R$ be a solution of the von Neumann equation \eqref{eq:VonNeumann} with $A$ and $\phi$ such that $B=\nabla \times A$, $\partial_t A$ and $\nabla \phi$ exist. Then the Weyl-Stratonovich transform $S$
of $R$, cf. \eqref{eq:weyl-stratonovich}, satisfies the magnetic von Neumann equation
\begin{align}
\begin{cases}
   \partial_t \tilde{S} =  \mathcal{D}_{1k} \mathcal{D}_{2k} \tilde{S}  + i y \cdot \mathcal{I}_2[E]\tilde{S}, \\ \tilde{S}|_{t=0} = \tilde{R}^{\mathrm{in}}\exp(\tfrac{i}{\hbar}\widetilde{\Gamma_{A^{\mathrm{in}}}}),
   \end{cases}
   \label{eq:magnetic von Neumann}
\end{align}
where $A^{\mathrm{in}} = A|_{t=0}$ and where
\begin{equation*}
    [\mathcal{D}_{1k}, \mathcal{D}_{2k}] = 0.
\end{equation*}
Let $W^{\hbar}_A = W^{\hbar}_A[R]$ be the magnetic Wigner transform of $R$. Then the magnetic Wigner equation is given by
\begin{align}
\label{eq:magnetic Wigner}
    (\partial_t +v\cdot \nabla_x) W^{\hbar}_A 
    &= \Phi^{\hbar}[E,B]W^{\hbar}_A,
\end{align}
where
\begin{align}
\begin{split}
    \Phi^{\hbar}[E,B]W^{\hbar}_A &:= i\hbar \nabla_v \cdot \left(\nabla_x \times (\Omega[B] W^{\hbar}_A) - \Omega[\nabla_x \times B] W^{\hbar}_A\right) \\ &\qquad +i\hbar \varepsilon_{jk\ell}\varepsilon_{nkm} \partial_{v_{\ell}} \partial_{v_{m}} \Omega[B_j]\Theta[B_n]W_A^{\hbar} \\
    &\qquad 
    -\nabla_v \cdot ((\Theta[E] + v \times \Theta[B]) W^{\hbar}_A).
\end{split}
\label{eq:Phi}
\end{align}
while $\Theta$ and $\Omega$ are defined in \eqref{eq:def theta} and $\varepsilon_{ijk}$ denotes the completely antisymmetric Levi-Civita symbol.
\end{proposition}
\begin{remark}
Note that $\varepsilon_{jk\ell}\varepsilon_{nkm} = \delta_{\ell m}\delta_{jn} - \delta_{\ell n} \delta_{jm}$. Then the term involving the second order derivative in $v$ in \eqref{eq:Phi} becomes
\begin{align*}
    &i\hbar (\delta_{\ell m}\delta_{jn} - \delta_{\ell n} \delta_{jm}) \partial_{v_{\ell}} \partial_{v_{m}} \Omega[B_j]\Theta[B_n]W_A^{\hbar}(x,v) \\&\qquad  = i\hbar\left(\Delta_v (\Omega[B] \cdot \Theta[B]W_A^{\hbar}(x,v)) -\partial_{v_j}\partial_{v_n}\Omega[B_j]\Theta[B_n]W_A^{\hbar}(x,v)\right).
\end{align*}
\end{remark}

\section{Velocity averaging for the magnetic Wigner equation}
\label{sec:VA}

We recall the following averaging lemma including a full derivative in $x$. Here, $\alpha$ denotes a multi-index and $h=(h_1,...,h_d)$ denotes a vector-valued function on $\mathbb{R}_t \times \mathbb{R}^d_x \times \mathbb{R}^d_{v}$.
\begin{theorem}[{\cite[Theorem 1]{perthameLimitingCaseVelocity1998},\cite[Theorem 1.7]{bouchutKineticEquationsAsymptotic2000}}]
\label{thm:ALPerthameSouganidis}
    Suppose that $f\colon \mathbb{R}_t \times \mathbb{R}^d_x \times \mathbb{R}^d_{v} \rightarrow \mathbb{R}$ solves the transport equation
    \begin{equation*}
        (\partial_t + v\cdot \nabla_x)f = \sum_{j,|\alpha|\leq m}\partial_{x_j}\partial_{v}^{\alpha}h_{j,\alpha},
    \end{equation*}
    Let $\varphi \in C^{m}_c(B_R)$. There exists $C\equiv C[R,\varphi]$ such that the averages in velocity
    \begin{equation*}
        \rho_{\varphi}(f)(t,x) := \int_{\mathbb{R}^{d}} f(t,x,v)\varphi(v) d v
    \end{equation*}
    satisfy the bound
    \begin{equation*}
        \|\rho_{\varphi}(f)\|_{L^p_{t,x}} \leq C \|f\|_{L^p_{t,x,v}}^{1-\frac{1}{(m+1)p'}}\|h\|_{L^p_{t,x,v}}^{\frac{1}{(m+1)p'}}, \qquad 1<p<\infty.
    \end{equation*}
\end{theorem}
We deduce the following variant of Theorem \ref{thm:ALPerthameSouganidis}, which describes the velocity averaging for a right-hand side of the form $\nabla_x \cdot h +g$, where $h=(h_1,...,h_d)$ is ``small" and where $g$, while not necessarily small, contains no derivative in $x$. More precisely, we consider an equation of the form
\begin{equation}
    (\partial_t +v\cdot \nabla_x)f = \sum_{j=1}^d\sum_{|\alpha|\leq m}\partial_{x_j}\partial_{v}^{\alpha} h_{j,\alpha} + \sum_{|\alpha| \leq m} \partial_v^{\alpha}g_{\alpha}.
    \label{eq:transport with T1}
\end{equation}
\begin{corollary}
\label{corollary AL}
    Let $m\geq 0$, $\varphi \in C^m_c(B_R)$ and $g,h\in L^2
    (\mathbb{R}_t,L^2(\mathbb{R}^d_x\times \mathbb{R}^d_v))$. Moreover, let $f \in L^2
    (\mathbb{R}_t,L^2(\mathbb{R}^d_x\times \mathbb{R}^d_v))$ be a solution of the transport equation \eqref{eq:transport with T1}. 
    There exists $C\equiv C[R,\varphi]$ such that the finite differences of the averages in velocity
    \begin{equation*}
        \delta_z \rho_{\varphi}(f) (x) := \rho_{\varphi}(f)(x+z)-\rho_{\varphi}(f)(x),
    \end{equation*}
    satisfy the  bound
    \begin{equation}
    \label{eq:bound modulus of continuity}
            \|\delta_z \rho_{\varphi}(f)\|_{L^2} \leq C \left(\|h\|_{L^2}+|z|\|g\|_{L^2}\right)^{\frac{1}{2(m+1)}} \|f\|_{L^2}^{1-\frac{1}{2(m+1)}}.
        \end{equation}
        In particular, let $f^n$ be a family of solutions to \eqref{eq:transport with T1}, with $h_j^n := T^{\varepsilon_n}_j f^n$, where $T_j^{\varepsilon_n}\colon L^2\rightarrow L^2$, $j=1,...,d$ is a family of bounded operators such that
    \begin{equation}
        \|h_j^n\|_{L^2} = \|T^{\varepsilon_n}_jf^n\|_{L^2} \leq \varepsilon_n\|f^n\|_{L^2}, \quad j=1,...,d.
        \label{eq:bound T}
    \end{equation}
    If $f^n$ and $g^n$ are bounded families in $L^2_{\mathrm{loc}}(\mathbb{R}_t,L^2(\mathbb{R}^d_x\times \mathbb{R}^d_v))$ and $\varepsilon_n$ converges to zero as $n\rightarrow \infty$, then the family $\rho^n_{\varphi} := \rho_{\varphi}(f^n)$ is relatively compact in $L^2([-T,T]\times B(0,R))$ for all $T>0$ and for all $R>0$.
    \begin{proof}
               Following the proof of \cite{perthameLimitingCaseVelocity1998}, we write \eqref{eq:transport with T1} in Fourier variables
        \begin{align*}
            \widehat{f}\left((\tau+v\cdot k)^2 + \beta^2 |k|^2\right) = \beta^2|k|^2\widehat{f} + \sum_{j=1}^d k_j(\tau+v\cdot k)\partial_v^{\alpha}\widehat{h_j} -i (\tau + v\cdot k)\partial_v^{\alpha}\widehat{g}
        \end{align*}
        Therefore we can write $f$ as
        \begin{equation*}
            f = f_0 + \sum_{j=1}^d f_j +f_{\ast},
        \end{equation*}
        with
        \begin{align*}
            \widehat{f_0} = \frac{\beta^2 |k|^2}{(\tau+v\cdot k)^2 + \beta^2 |k|^2}\widehat{f}, && \widehat{f_j} = \frac{k_j(\tau+v\cdot k)}{(\tau+v\cdot k)^2 + \beta^2 |k|^2}\partial_v^{\alpha}\widehat{h_j}, 
        \end{align*}
        and
        \begin{equation*}
            \widehat{f_{\ast}} = \frac{-i(\tau +v\cdot k)}{(\tau+v\cdot k)^2 + \beta^2 |k|^2}\partial_v^{\alpha}\widehat{g}.
        \end{equation*}
        Moreover, let 
        \begin{align*}
            \rho_0(t,x) = \int_{\mathbb{R}^d} \varphi(v) f_0(t,x,v) d v, && \rho_j(t,x) = \int_{\mathbb{R}^d} \varphi(v) f_j(t,x,v) d v, 
            \end{align*}
            and 
            \begin{align*}\rho_{\ast}(t,x) = \int_{\mathbb{R}^d} \varphi(v) f_{\ast}(t,x,v) d v.
        \end{align*}
        Then the average $\rho_{\varphi}(f)$ can be written as
        \begin{equation}
            \rho_{\varphi}(f) = \rho_0 + \sum_{j=1}^d \rho_j + \rho_{\ast}.
        \end{equation}
        The averages $\rho_0$ and $\rho_j$ are treated in \cite[Lemma 3, 4]{perthameLimitingCaseVelocity1998}, together with the following estimate
        \begin{equation}
            \label{eq:estimate perthame souganidis}\|\rho_0+\sum_{j=1}^d\rho_j\|_{L^2(\mathbb{R}_t\times \mathbb{R}^d_x)} \leq C(\|\varphi\|_{W^{m,\infty}},R) \left(\beta^{\frac{1}{2}}\|f\|_{L^2} + {\beta^{-m-\frac{1}{2}}}\|h\|_{L^2}\right),
        \end{equation}
        where $R>0$ is such that $\mathrm{supp}(\varphi) \subset B_R(0)$.
        Therefore it remains to obtain an estimate for $\rho_{\ast}$. We have
        \begin{align*}
            \widehat{\rho_{\ast}}(\tau,k) = \int_{\mathbb{R}^d} \varphi(v) \widehat{f_{\ast}}(\tau,k,v) dv = \int_{\mathbb{R}^d} \varphi(v) \frac{-i(\tau +v\cdot k)}{(\tau+v\cdot k)^2 + \beta^2 |k|^2}\partial_v^{\alpha}\widehat{g} d v.
        \end{align*}
        Suppose that $m=0$ so that $|\alpha|=0$. Together with the change of variables $\gamma = \tau/\beta|k|$, $v_1 = v\cdot k/|k|$ we have
        \begin{align*}
            |\widehat{\rho_{\ast}}(\tau,k)|^2 &\leq \|\varphi\|_{L^\infty}^2 \int_{\mathbb{R}^d}|\widehat{g}|^2 d v  \int_{B(0,R)} \frac{|\tau +v\cdot k|^2}{|(\tau+v\cdot k)^2 + \beta^2 |k|^2|^2} d v \\
            &= \frac{1}{|k|^2}\frac{\|\varphi\|_{L^\infty}^2}{\beta^2} \int_{\mathbb{R}^d}|\widehat{g}|^2 d v\int_{B(0,R)} \frac{\left(\gamma + \frac{v_1}{\beta}\right)^2}{\left((\gamma + \frac{v_1}{\beta})^2 + 1\right)^2} d v_1 .
        \end{align*}
        The second integral is proportional to $\beta$. This yields
        \begin{equation*}
            |\widehat{\rho_{\ast}}(\tau,k)|^2 \leq  \frac{\|\varphi\|_{L^\infty}^2}{\beta|k|^2} \int_{\mathbb{R}^d}|\widehat{g}|^2 d v.
        \end{equation*}
        This bound is interesting for $|k|>1$. For low frequencies, i.e. $|k|\le 1$, one has the straightforward inequality
        \begin{equation*}
            \|\rho_{\ast}\|_{L^2_{t,x}} = \left\| \int_{\mathbb{R}^d}\varphi(v)f_{\ast}(t,x,v) dv\right\|_{L^2_{t,x}} \leq C'(R)\|\varphi\|_{L^\infty} \|f_{\ast}\|_{L^2_{t,x,v}}.
        \end{equation*}
For $m>0$ we denote
         \begin{equation}
             Q_{\beta}(\tau,k,v) := \frac{\tau +v\cdot k}{(\tau+v\cdot k)^2 + \beta^2 |k|^2}.
         \end{equation}
        After integrating by parts and applying the Cauchy-Schwarz we obtain, with $|\alpha| \leq m$,
        \begin{align*}
           |\widehat{\rho_{\ast}}(\tau,k)|^2 &= \left|\int_{\mathbb{R}^d} \varphi Q_{\beta}\partial_v^{\alpha}\widehat{g} d v\right|^2 \\ &\leq \|\varphi\|_{W^{m,\infty}_v}^2 \int_{\mathbb{R}^d}|\widehat{g}|^2d v\int_{B(0,R)}\left[|Q_{\beta}|^2+
           \frac{|k|^{2m}|\tau +v\cdot k|^{2(m+1)}}{|(\tau+v\cdot k)^2 + \beta^2 |k|^2|^{2(m+1)}}\right]d v.
        \end{align*}
        Again, changing variables to $\gamma = \tau/\beta|k|$, $v_1 = v\cdot k/|k|$ yields
        \begin{align*}
            |\widehat{\rho_{\ast}}(\tau,k)|^2 &\leq \|\varphi\|_{W^{m,\infty}_v}^2\|\widehat{g}\|_{L^2_v}^2 \int_{\mathbb{R}} \left[\frac{1}{\beta^2|k|^2}\frac{\left(\gamma + \frac{v_1}{\beta}\right)^2}{\left((\gamma + \frac{v_1}{\beta})^2 + 1\right)^2} + \frac{1}{\beta^{2(m+1)}|k|^2}\frac{\left(\gamma + \frac{v_1}{\beta}\right)^{2(m+1)}}{\left((\gamma + \frac{v_1}{\beta})^2 + 1\right)^{2(m+1)}} \right]dv_1.
        \end{align*}
        Again, the integrals with respect to $v_1$ are proportional to $\beta$ and thus
        \begin{align}
            \label{eq:first estimate rho*}|\widehat{\rho_{\ast}}(\tau,k)|^2 &\leq  \|\varphi\|_{W^{m,\infty}_v}^2\|\widehat{g}\|_{L^2_v}^2 \left[\frac{1}{\beta|k|^2}+\frac{1}{\beta^{2m+1}|k|^2}\right].
        \end{align}
       
        Now we obtain an estimate in $L^2(\mathbb{R}^d)$ for the finite differences
        \begin{equation*}
            \delta_zf(x) = f(x+z)-f(x),
        \end{equation*}
         We have
        \begin{equation*}
            \delta_z \rho_{\varphi}(f) := \delta_z \rho_0 + \sum_{j=1}^d\delta_z \rho_j + \delta_z \rho_{\ast}.
        \end{equation*}
        For the first two terms we have the following estimates, due to \eqref{eq:estimate perthame souganidis},
        \begin{align}
            \label{eq:estimate rho0}\|\delta_z\rho_0\|_{L^2} &\leq 2 \|\rho_0\|_{L^2} \leq C(\|\varphi\|_{L^{\infty}},R)\beta^{\frac{1}{2}}\|f\|_{L^2},\\
            \label{eq:estimate rhoj}\sum_{j=1}^d\left\|\delta_z\rho_j\right\|_{L^2} &\leq 2 \sum_{j=1}^d\left\|\rho_j\right\|_{L^2} \leq C(\|\varphi\|_{W^{m,\infty}},R)\beta^{-m-\frac{1}{2}}\|h\|_{L^2}.
        \end{align}
        For $\delta_z \rho_{\ast}$ we observe that
        \begin{equation*}
            \widehat{\delta_z \rho_{\ast}}(\tau,k) = \left(e^{-ik\cdot z}-1\right)\widehat{\rho_{\ast}}(\tau,k),
        \end{equation*}
        from which we deduce
        \begin{equation*}
            |\widehat{\delta_z \rho_{\ast}}(\tau,k)| = 2\left|\sin(\tfrac{k\cdot z}{2})\right||\widehat{\rho_{\ast}}(\tau,k)| \leq |z||k||\widehat{\rho_{\ast}}(\tau,k)|.
        \end{equation*}
        It follows that
        \begin{equation*}
            \|\delta_z \rho_{\ast}\|_{L^2_{t,x}} \leq |z| \left\||k|\widehat{\rho_{\ast}}(\tau,k)\right\|_{L^2_{\tau,k}}.
        \end{equation*}
        From \eqref{eq:first estimate rho*} we obtain
        \begin{equation*}
            \||k|\widehat{\rho_{\ast}}(\tau,k)\|_{L^2_{\tau,k}} = \left(\int\left|\int_{\mathbb{R}^d} |k|^2 |\widehat{\rho_{\ast}}(\tau,k)|^2dk\right|d\tau\right)^{\frac{1}{2}}\leq C(R) \|{g}\|_{L^2_{t,x,v}} \left[\beta^{-\frac{1}{2}}+ \beta^{-m-\frac{1}{2}}\right],
        \end{equation*}
        while
        \begin{equation}
        \label{eq:estimate rho*}
            \|\delta_z \rho_{\ast}\|_{L^2_{t,x}} \leq C(R) |z|  \left[\beta^{-\frac{1}{2}}+ \beta^{-m-\frac{1}{2}}\right]\|{g}\|_{L^2_{t,x,v}}.
        \end{equation}
        Combining \eqref{eq:estimate rho0}, \eqref{eq:estimate rhoj} and \eqref{eq:estimate rho*} yields
        \begin{equation*}
            \|\delta_z \rho_{\varphi}(f)\|_{L^2} \leq C\left[ \beta^{\frac{1}{2}}\|f\|_{L^2} +\beta^{-m-\frac{1}{2}}\|h\|_{L^2} +\left(\beta^{-\frac{1}{2}}+ \beta^{-m-\frac{1}{2}}\right)|z|\|{g}\|_{L^2} \right],
        \end{equation*}
        where $C$ designates a constant depending on $R$, $\|\varphi\|_{W^{m,\infty}}$ and $d$. For $\beta <1$, the dominant term for the term involving $\|g\|_{L^2}$ is $\beta^{-m-\frac{1}{2}}$. Optimizing yields
        \begin{equation}
            \beta = \left(\frac{\|h\|_{L^2}+|z|\|g\|_{L^2}}{\|f\|_{L^2}}\right)^{\frac{1}{m+1}}.
        \end{equation}
        For $\beta \geq 1$ we have the obvious inequality
        \begin{equation*}
            \|\delta_z\rho_{\varphi}(f)\|_{L^2} \leq C\|f\|_{L^2}.
        \end{equation*}
        Thus,
        \begin{equation}
        \label{eq:bound modulus of continuity proof}
            \|\delta_z \rho_{\varphi}(f)\|_{L^2} \leq C \left(\|h\|_{L^2}+|z|\|g\|_{L^2}\right)^{\frac{1}{2(m+1)}} \|f\|_{L^2}^{1-\frac{1}{2(m+1)}}.
        \end{equation}
        This proves \eqref{eq:bound modulus of continuity}.

        For the statement about the relative compactness of $\rho^n_{\varphi}$ we take $h^n_j = T^{\varepsilon_n}_j f^n$ and observe that from the equation,
        \begin{align*}
            \partial_t \rho^n_{\varphi} &= -\mathrm{div}_x\int_{\mathbb{R}^d}\varphi(v)(vf^n)dv +\mathrm{div}_x\int_{\mathbb{R}^d} \varphi \sum_{|\alpha|\leq m}\partial_v^{\alpha}T^{\varepsilon_n}f^ndv + \int_{\mathbb{R}^d}\varphi \sum_{|\alpha|\leq m}\partial_v^{\alpha}g_{\alpha}^ndv \\
            &= -\mathrm{div}_x\int_{\mathbb{R}^d}(v\varphi(v))f^ndv +\mathrm{div}_x\int_{\mathbb{R}^d} \sum_{|\alpha|\leq m}(-1)^{|\alpha|}\partial_v^{\alpha}\varphi(v) T^{\varepsilon_n}f^ndv \\ &\qquad + \int_{\mathbb{R}^d}\sum_{|\alpha|\leq m}(-1)^{|\alpha|}\partial_v^{\alpha}\varphi(v)g^n_{\alpha}dv.
        \end{align*}
        Thus, we observe easily that, for $T >0$, 
        \begin{equation}
            \sup_n \|\partial_t \rho^n_{\varphi}\|_{L^2([-T-1,T+1],H^{-1}_x)} < \infty.
        \end{equation}
        The compactness of the sequences $\rho^{n}_{\varphi}$ now follows from an application of the Kolmogorov-Riesz theorem which we briefly sketch. According to Kolmogorov-Riesz (see for instance  Theorem 4.26 of \cite{brezisFunctionalAnalysisSobolev2011}), a bounded family $\mathcal{F}$ in $L^p(\mathbb{R}^N)$  is relatively compact in $L^p_{\mathrm{loc}}(\mathbb{R}^N)$ if 
        \begin{equation}
            \|\delta_z f \|_{L^p(\mathbb{R}^N)} \rightarrow 0 \quad \text{as }|z|\rightarrow 0 \quad \text{uniformly for } f\in \mathcal{F}. 
        \end{equation}
        In this case we choose $N=1+d$ and $z=(s,y) \in \mathbb{R}^{N}$. First, we have that the family $\rho^n_{\varphi}$ is uniformly bounded in $L^2((-T-1,T+1),\mathbb{R}^d)$ by the $L^2$ bound of $f^n$. The continuity of space translations has already been established. Therefore it remains to treat the time translations. To that end we introduce the function
        \begin{equation}
            w^n(t,x) := \chi(t) \rho^n_{\varphi}(t,x), 
        \end{equation}
        where $\chi$ is a cut-off function in $t$ that is zero outside of $[-T-1,T+1]$, one on $[-T,T]$ and smooth. Moreover, denote by $\phi_{\eta}$, $\eta >0 $ a standard mollifier in space. We write finite differences of $w^n$ in time as
        \begin{align*}
            |w^n(t+s,x)-w^n(t,x)| &\leq |w^n(t+s,x)-\phi_{\eta}\ast w^n(t+s,x)| + |w^n(t,x)-\phi_{\eta}\ast w^n(t,x)| \\ &\qquad +|\phi_{\eta} \ast w^n(t+s,x)-\phi_{\eta}\ast w^n(t,x)|.
        \end{align*}
        By the translation invariance of $L^p$-norms, the first two terms on the right-hand side can be grouped together and it remains to estimate
        \begin{align*}
            \|w^n- \phi_{\eta} \ast w^n\|_{L^2}
        \quad
        \text{and}
        \quad
            \|\phi_{\eta} \ast w^n(\cdot + s)-\phi_{\eta} \ast w^n\|_{L^2}.
        \end{align*}
        The former is simply treated by the bound on the space translations already established, which reads
        \begin{equation*}
            \|\rho^n_{\varphi}(t,\cdot+y)-\rho^n_{\varphi}(t,\cdot)\|_{L^2} \leq C (\varepsilon_n^{\theta} + |y|^{\theta}),
        \end{equation*}for some $\theta >0$. For the latter, we observe that
        \begin{equation*}
            \partial_t w^n(t,x) = \chi'(t)\rho^n_{\varphi}(t,x) + \chi(t) \partial_t \rho^n_{\varphi}(t,x).
        \end{equation*}
        Moreover,
        \begin{equation*}
            \|\phi_{\eta} \ast g \|_{L^2} \leq \frac{C}{\eta} \|g\|_{H^{-1}}.
        \end{equation*}
        Thus, by the bound on $\partial_t \rho^{n}_{\varphi}$,
        \begin{equation*}
            \|\partial_t(\phi_{\eta} \ast w^{n})\|_{L^2(\mathbb{R}\times \mathbb{R}^d)} \leq C\left(1+\frac{1}{\eta}\right).
        \end{equation*}
        Finally,
        \begin{align*}
            \|\phi_{\eta}\ast w^{n}(t+s,x)- \phi_{\eta}\ast  w^n(t,x)\|_{L^2} = \left\|\int_0^s (\partial_t(\phi_{\eta}\ast w^{n}))(t+\sigma,x)d\sigma \right\|_{L^2} \leq C|s|\left(1+\frac{1}{\eta}\right).
        \end{align*}
        By choosing an appropriate $\eta$, this difference goes to zero as $|s|$ goes to zero. Therefore continuity of the time translations is established and we obtain relative compactness of the family $w^n$ on $L^2((-T-1,T+1)\times B(0,R))$ and consequently of $\rho^{n}_{\varphi}$ in $L^2((-T,T)\times B(0,R))$ since $\rho^{n}_{\varphi} = w^n$ on $(-T,T)\times \mathbb{R}^d$.
    \end{proof}
\end{corollary}

Now we proceed to the proof of Theorem \ref{thm:main}. For simplicity, we omit the $\hbar$-superscripts of $E^{\hbar}$ and $B^{\hbar}$. Recall that now $d=3$.
\begin{proof}[Proof of Theorem \ref{thm:main}]

By Proposition \ref{thm:condition mixed states} $W^{\hbar}_A$ is a bounded family in $L^2_{x,v}$. We apply Corollary \ref{corollary AL} with $\varepsilon_n=\hbar_n$. We will abbreviate $\hbar=\hbar_n$ for notational simplicity. Recall from Proposition \ref{thm:magnetic equations}, \eqref{eq:magnetic Wigner}-\eqref{eq:Phi} the magnetic Wigner equation
\begin{align*}
     (\partial_t +v\cdot \nabla_x) W^{\hbar}_A 
    &= \Phi^{\hbar}[E,B]W^{\hbar}_A,
\end{align*}
where
\begin{align}
\label{eq:Phi2}
\begin{split}
    \Phi^{\hbar}[E,B]W^{\hbar}_A &:= i\hbar \nabla_v \cdot \left(\nabla_x \times (\Omega[B] W^{\hbar}_A) - \Omega[\nabla_x \times B] W^{\hbar}_A\right) \\ &\qquad +i\hbar \varepsilon_{jk\ell}\varepsilon_{nkm} \partial_{v_{\ell}} \partial_{v_{m}} \Omega[B_j]\Theta[B_n]W_A^{\hbar} \\
    &\qquad - \nabla_v \cdot ((\Theta[E] + v \times \Theta[B]) W^{\hbar}_A).
    \end{split}
\end{align}
while we recall from \eqref{eq:def theta} that the pseudo-differential operators $\Omega,\Theta$ are defined as
\begin{align}
    \Omega[f]\psi &:= \tfrac{1}{(2\pi)^d}\mathcal{F}_{y\rightarrow v}\left[\mathcal{I}_1[f] \right]\ast_v \psi, & \mathcal{I}_1[f]&:= \int_{-1/2}^{1/2} f(x+s\hbar y) s d s, \label{eq:Pseudodiff1}\\
     \Theta[f]\psi &:= \tfrac{1}{(2\pi)^d}\mathcal{F}_{y\rightarrow v}\left[\mathcal{I}_2[f]\right]\ast_v \psi, &   \mathcal{I}_2[f]&:= \int_{-1/2}^{1/2} f(x+\sigma \hbar y) d \sigma.\label{eq:Pseudodiff2}
\end{align}
Therefore we deduce that the operator $\Phi^{\hbar}[E,B]$ can be written as the following sum
\begin{equation*}
    \Phi^{\hbar}[E,B] = \sum_{j,k=1}^d\partial_{x_k} \partial_{v_j} T^{(1)}_{jk} + \sum_{j=1}^d \partial_{v_j}\left(\sum_{k=1}^d\partial_{v_k}T^{(2)}_{jk} +T^{(3)}_j + T^{(4)}_j+ T^{(5)}_j\right),
\end{equation*}
where $T^{(k)}$, $k=1,...,5$ are the operators
\begin{align*}
\begin{cases}
    T^{(1)}_{jk} &:= i\hbar \epsilon_{jk\ell} \Omega[B_{\ell}], \\
    T^{(2)}_{jk} &:= i\hbar \epsilon_{\ell n j}\epsilon_{mnk}  \Omega[B_\ell]\Theta[B_m],
    \\T^{(3)}_j &:= -i\hbar \epsilon_{jk\ell}\Omega[\partial_{x_k}B_{\ell}],\\
    T^{(4)}_j &:= -\epsilon_{jk\ell}v_k\Theta[B_\ell],\\
    T^{(5)}_j &:=  -\Theta[E_j].
    \end{cases}
\end{align*}
Note that $T^{(k)}$, $k=1,...,5$ are pseudo-differential operators (acting on $W^{\hbar}_A$ by convolution) which are defined via the pseudo-differential operators $\Omega$ and $\Theta$, defined in \eqref{eq:Pseudodiff1}-\eqref{eq:Pseudodiff2}.

Thus in Corollary \ref{corollary AL} we take
\begin{align}
\label{eq:h and g}
    h = T^{(1)}f, && g = \nabla_v\cdot(\nabla_v \cdot T^{(2)} + T^{(3)} + T^{(4)} + T^{(5)})f.
\end{align}
with $f=W^{\hbar}_A$.
In order to apply Corollary \ref{corollary AL} we then need to show that $T^{(1)}$ satisfies the bound
\begin{equation}
    \|T^{(1)}f\|_2 \leq C\hbar \|f\|_2,
\end{equation}
and that $g$ is bounded in $L^2_xH^{-2}_{v}$.

For 
$T^{(1)}$ we have to show that
\begin{equation*}
    \|T^{(1)}f\|_{L^2} \leq C\hbar \|f\|_{L^2} \quad \Leftrightarrow \quad \|\Omega[B]f\|_{L^2} \leq C \|f\|_{L^2}.
\end{equation*}
By definition, we have
\begin{align*}
    \left\|\Omega[B] f\right\|_{L^2_{x,v}} =\left\|\tfrac{1}{(2\pi)^d}\mathcal{F}_{y\rightarrow v}\left[\int_{-1/2}^{1/2}B(x+s \hbar y)sd s \right]\ast_v f\right\|_{L^2_{x,v}}.
\end{align*}
Now we use the elementary inequality
\begin{equation}
    \|f\ast g\|_{L^2} \leq \|\mathcal{F}f\|_{L^\infty}\|g\|_{L^2}.
    \label{eq:convolution inequality}
\end{equation}
This yields
\begin{align*}
    \left\|\Omega[ B] f\right\|_{L^2_{x,v}} &\leq \sup_{x,z}\left|\tfrac{1}{(2\pi)^d}\mathcal{F}_{v\rightarrow z}\left[\mathcal{F}_{y\rightarrow v}\left[\int_{-1/2}^{1/2}B(x+s \hbar y)sd s \right]\right] \right|\|f\|_{L^2_{x,v}}\\
    &\leq\sup_{x,z} \left|\int_{-1/2}^{1/2}B(x+s \hbar z)sd s\right|\|f\|_{L^2_{x,v}} \\
    &\leq \sup_{x,z,s} \left|B(x+s \hbar z)\right| \int_{-1/2}^{1/2}|s| d s \|f\|_{L^2_{x,v}} \\
    &\leq \tfrac{1}{4} \| B\|_{L^{\infty}} \|f\|_{L^2_{x,v}}.
\end{align*}

For $T^{(2)}$ we have to show that
\begin{equation*}
     \|\Omega[B]\Theta[B]f\|_{L^2} \leq C\|f\|_{L^2}.
\end{equation*}
From the previous term we know that $\|\Omega[B]f\|_{L^2} \leq  \frac{1}{4}\|B\|_{L^{\infty}} \|f\|_{L^2}$. Therefore it remains to bound $\Theta[B]$. Similar to the above we observe
\begin{align*}
    \|\Theta[B]f\|_{L^2_{x,v}} &\leq  \sup_{x,z,\sigma}\left|B(x+\sigma \hbar z)\right| \int_{-1/2}^{1/2} d \sigma \|f\|_{L^2_{x,v}} \\ &\leq \|B\|_{L^{\infty}}\|f\|_{L^2_{x,v}}.
\end{align*}
Thus
\begin{equation*}
     \|\Omega[B]\Theta[B]f\|_{L^2_{x,v}} \leq \tfrac{1}{4}\|B\|_{L^{\infty}}^2\|f\|_{L^2_{x,v}}.
\end{equation*}

For $T^{(3)}$ we repeat the argument for $T^{(1)}$ with $B$ replaced by $\nabla \times B$ and obtain
\begin{equation*}
    \|\Omega[\nabla \times B]f\|_{L^2_{x,v}} \leq \|\nabla \times B\|_{L^{\infty}}\|f\|_{L^2_{x,v}}.
\end{equation*}
However, since $T^{(3)}$ scales with $\hbar$, it is sufficient to assume that
\begin{equation}
    \|\hbar (\nabla \times B)\|_{L^{\infty}} \leq C.
\end{equation}

For $T^{(4)}$ we observe that the multiplication by $v$ can be absorbed into the cutoff function $\varphi \in C^2_c(\mathbb{R}^3_v)$\footnote{If $\varphi \in C^2_c$ then $\chi$ defined by $\chi(v) = v\varphi(v)$ is also in $C^2_c$. Therefore in the proof of the averaging lemma in Corollary \ref{corollary AL} one can replace $\varphi$ by $\chi$.}. Therefore the term to estimate in $T^{(4)}$ reduces to $\Theta[B]$, which was already done for $T^{(2)}$. Note that here we do not need to extract any power of $\hbar$ since $T^{(4)}$ is contained in the part of $\Phi^{\hbar}$ that does not contain any derivatives in $x$. 

We proceed similarly for $T^{(5)}$, for which the following estimate suffices.
\begin{align*}
    \|\Theta[E]f\|_{L^2_{x,v}} &\leq \sup_{x,z,\sigma}|E(x+\sigma \hbar z)| \int_{-1/2}^{1/2}d \sigma \|f\|_{L^2_{x,v}}\\
    &\leq \|E\|_{L^\infty} \|f\|_{L^2_{x,v}}.
\end{align*}
Thus we have shown that
\begin{equation}
    g \in L^2_{x}H^{-2}_v,
\end{equation} and that
\begin{equation}
    \|T^{(1)}f\|_{L^2} \leq C \hbar \|B\|_{L^{\infty}}\|f\|_{L^2}.
\end{equation}
Therefore we can invoke Corollary \ref{corollary AL} to $f= W^{\hbar}_A$ with $h,g$ given by \eqref{eq:h and g}. We note that $W^{\hbar}_A \in L^{\infty}(\mathbb{R}_t,L^2(\mathbb{R}^3_x \times \mathbb{R}^3_v))$ by virtue of Proposition \ref{thm:condition mixed states} and \eqref{eq:bound Wigner L^2} on account of the mixed state condition \eqref{eq:condition mixed states}. Since $W^{\hbar}_A \in L^{\infty}_t L^2_{x,v}$ as opposed to $L^{2}_t L^2_{x,v}$ required for the Corollary we need choose an appropriate time cut-off similarly to the proof of Corollary \ref{corollary AL} (we leave out the details) to conclude that the family $\rho_{\varphi}(W^{\hbar}_A)$ is relatively compact in $L^2([-T,T] \times B(0,R))$.
\end{proof}

\appendix

\section{Proof of Proposition \ref{thm:magnetic equations}}
\label{app:derivation magnetic equations}
In this Section we again assume that $d=3$.
\begin{proof}Consider the magnetic von Neumann equation with the magnetic Hamiltonian $\tfrac{1}{2}|V|^2+ \phi$, where $\phi$ denotes the multiplication operator by $\phi(X)$. Note that we use Einstein's summation convention where summation over repeated indices is implied.
\begin{equation}
i\hbar\partial_tR(t)=[\frac12 V_jV_j+ \phi,R]\,,\quad V_j=-i\hbar\partial_j-A_j
\end{equation}
Rewrite it in terms of the operator kernel of $R$:
\begin{align}
\begin{split}
i\hbar\partial_tR(t,X,Y)&=\frac12\left((-i\hbar\partial_{X_j}-A_j(X))^2-(i\hbar\partial_{Y_j}-A_j(Y))^2\right)R(t,X,Y) \\
&\qquad +(\phi(X)-\phi(Y))R(t,X,Y).
\end{split}
\end{align}
Since $[-i\hbar\partial_{X_j}-A_j(X),i\hbar\partial_{Y_j}-A_j(Y)]=0$, this equation is recast as
\begin{align*}
i\hbar\partial_tR(t,X,Y)=&\frac12(-i\hbar\partial_{X_j}-A_j(X)+i\hbar\partial_{Y_j}-A_j(Y))
\\
&\times(-i\hbar\partial_{X_j}-A_j(X)-i\hbar\partial_{Y_j}+A_j(Y))R(t,X,Y)\\
&\qquad +(\phi(X)-\phi(Y))R(t,X,Y).
\end{align*}
In the Weyl variables $(x,y)$, cf. \eqref{eq:weyl variables}, the derivatives become
\begin{equation*}
\partial_{X_j}=\tfrac12\partial_{x_j}+\tfrac1\hbar\partial_{y_j},\quad\partial_{Y_j}=\tfrac12\partial_{x_j}-\tfrac1\hbar\partial_{y_j}.
\end{equation*}
Hence
\begin{align}
\partial_t\tilde R(t,x,y) &=(i\partial_{y_j}+\tfrac12(A_j(x+\tfrac\hbar{2}y)+A_j(x-\tfrac\hbar{2}y)))
\nonumber \\
&\quad \times(\partial_{x_j}-\tfrac{i}\hbar(A_j(x+\tfrac\hbar{2}y)-A_j(x-\tfrac\hbar{2}y)))\tilde R(t,x,y) \nonumber \\
&\qquad +\tfrac{1}{i\hbar}(\phi(x+\tfrac{\hbar y}{2})-\phi(x-\tfrac{\hbar y}{2}))\tilde{R}(t,x,y) \nonumber \\
&=: D_{1j} D_{2j} \tilde R(t,x,y) + \tfrac{1}{i\hbar}(\phi(x+\tfrac{\hbar y}{2})-\phi(x-\tfrac{\hbar y}{2}))\tilde{R}(t,x,y).
\end{align}
where
\begin{align}
    D_{1j} &:= i\partial_{y_j}+\tfrac12(A_j(x+\tfrac\hbar{2}y)+A_j(x-\tfrac\hbar{2}y)), \\
    D_{2j} &:= \partial_{x_j}-\tfrac{i}\hbar(A_j(x+\tfrac\hbar{2}y)-A_j(x-\tfrac\hbar{2}y)).
\end{align}
The equation for $\tilde R$ in terms of $\tilde{S}$ (recall that $S$ is defined in \eqref{eq:weyl-stratonovich_tilde}) is recast as
\begin{align}
\begin{split}
\partial_t\tilde S(t,x,y)&=\left(e^{\frac{i}\hbar\widetilde{\Gamma_A}[x,y]}D_{1j}e^{-\frac{i}\hbar\widetilde{\Gamma_A}[x,y]}\right)
\left(e^{\frac{i}\hbar\widetilde{\Gamma_A}[x,y]}D_{2j}e^{-\frac{i}\hbar\widetilde{\Gamma_A}[x,y]}\right)\tilde S(t,x,y) \\
&\qquad + \tfrac{1}{i\hbar}(\phi(x+\tfrac{\hbar y}{2})-\phi(x-\tfrac{\hbar y}{2}))\tilde{S}(t,x,y) + \tfrac{i}{\hbar} \left(\partial_t \widetilde{\Gamma_A}[x,y]\right)\tilde{S}(x,y).
\end{split}
\end{align}
We will show that the operators $D_{1j}$ and $D_{2j}$ conjugated by $\exp(\frac{i}{\hbar}\widetilde{\Gamma_A}[x,y])$ coincide with $\mathcal{D}_{1}$ and $\mathcal{D}_2$ (defined in \eqref{eq:DiffOpCurved1}-\eqref{eq:DiffOpCurved2}), i.e.
\begin{align*}
\mathcal{D}_{1j} &= e^{\frac{i}\hbar\widetilde{\Gamma_A}[x,y]}D_{1j}e^{-\frac{i}\hbar\widetilde{\Gamma_A}[x,y]},
\\
\mathcal{D}_{2j} &= e^{\frac{i}\hbar\widetilde{\Gamma_A}[x,y]}D_{2j}e^{-\frac{i}\hbar\widetilde{\Gamma_A}[x,y]}.
\end{align*}
It is more convenient to rewrite them in the original variables $X,Y$:
\begin{align}
\mathcal{D}_{1j} =& e^{\frac{i}\hbar\Gamma_A[X,Y]}(i\frac\hbar{2}(\partial_{X_j}-\partial_{Y_j})+\frac12(A_j(X)+A_j(Y)))e^{-\frac{i}\hbar\Gamma_A[X,Y]}
\nonumber \\
=& i\frac\hbar{2}(\partial_{X_j}-\partial_{Y_j})+\frac12(A_j(X)\!+\!A_j(Y))+\frac12(\partial_{X_j}-\partial_{Y_j})\Gamma_A[X,Y], \label{eq: cal D1j original variables}
\end{align}
and
\begin{align}
\mathcal{D}_{2j}=& e^{\frac{i}\hbar\Gamma_A[X,Y]}(\partial_{X_j}\!+\!\partial_{Y_j}\!-\!\tfrac{i}\hbar(A_j(X)\!-\!A_j(Y)))e^{-\frac{i}\hbar\Gamma_A[X,Y]}\nonumber
\\
=& \partial_{X_j}+\partial_{Y_j}-\frac{i}\hbar(A_j(X)-A_j(Y))-\frac{i}\hbar(\partial_{X_j}+\partial_{Y_j})\Gamma_A[X,Y] .\label{eq: cal D2j original variables}
\end{align}
We have
\[
\begin{aligned}
\partial_{X_k}\Gamma_A[X,Y]=&\partial_{X_k}\int_0^1A_\ell(X+t(Y-X))(Y_\ell-X_\ell)dt
\\
=&\int_0^1\partial_kA_\ell(X+t(Y-X))(Y_\ell-X_\ell)(1-t)dt
\\
&-\int_0^1A_\ell(X+t(Y-X))\delta_{k\ell}dt.
\end{aligned}
\]
Similarly,
\[
\begin{aligned}
\partial_{Y_k}\Gamma_A[X,Y]=&\partial_{Y_k}\int_0^1A_\ell(X+t(Y-X))(Y_\ell-X_\ell)dt
\\
=&\int_0^1\partial_kA_\ell(X+t(Y-X))(Y_\ell-X_\ell)tdt
\\
&+\int_0^1A_\ell(X+t(Y-X))\delta_{k\ell}dt.
\end{aligned}
\]
Hence
\begin{align}
\begin{split}
\frac{1}{2}(\partial_{X_k}-\partial_{Y_k})\Gamma_A[X,Y] =&\frac{1}{2}\int_0^1\partial_kA_\ell(X+t(Y-X))(Y_\ell-X_\ell)(1-2t)dt
\\
&\qquad -\int_0^1A_{k}(X+t(Y-X))dt,
\end{split}
\label{eq:Derivative(dX-dY)GammaA}
\end{align}
while
\begin{align}
(\partial_{X_k}+\partial_{Y_k})\Gamma_A[X,Y]=&\int_0^1\partial_kA_\ell(X+t(Y-X))(Y_\ell-X_\ell)dt=\Gamma_{\partial_kA}(X,Y).
\label{eq:Derivative(dX+dY)GammaA}
\end{align}\\
\underline{The operator $\mathcal{D}_{1j}$}
In \eqref{eq:Derivative(dX-dY)GammaA}, we change variables to $s=\frac{1}{2}-t$ and obtain:
\begin{align*}
\begin{split}
\frac{1}{2}(\partial_{X_k}-\partial_{Y_k})\Gamma_A[X,Y] =&\int_{-1/2}^{1/2}\partial_kA_\ell(\tfrac{1}{2}(X+Y)+s(X-Y))(Y_\ell-X_\ell)sds
\\
&\qquad -\int_{-1/2}^{1/2}A_{k}(\tfrac{1}{2}(X+Y)+s(X-Y))ds.
\end{split}
\end{align*}
In the Weyl variables $x,y$ (recall \eqref{eq:weyl variables}) we obtain
\begin{align}
\frac{1}{2}(\partial_{X_k}-\partial_{Y_k})\Gamma_A[X,Y]
=&-\hbar y_\ell \int_{-1/2}^{1/2} \partial_kA_\ell (x+\hbar sy) s d s - \int_{-1/2}^{1/2} A_k(x+\hbar sy) d s.
\label{eq:Derivative(dX-dY)GammaAWeyl}
\end{align}
The last term on the right-hand side of \eqref{eq:Derivative(dX-dY)GammaAWeyl} can be rewritten as
\begin{equation}
\int_{-1/2}^{1/2} A_k(x+\hbar sy) d s = \int_{0}^{1/2} A_k(x+\hbar sy) d s + \int_{-1/2}^{0} A_k(x+\hbar sy) d s.
\label{eq:PropositionAux1}
\end{equation}
On the other hand we write
\begin{align*}
\frac{1}{2}(A_k(X)+A_k(Y)) = \int_0^{1/2} A_k(X)d s + \int_{-1/2}^0A_k(Y) d s.
\end{align*}
In Weyl variables this is written as
\begin{align}
\frac{1}{2}(A_k(x+\tfrac{\hbar}{2}y)+A_k(x-\tfrac{\hbar}{2}y)) = \int_0^{1/2} A_k(x+\tfrac{\hbar}{2}y) d s+\int_{-1/2}^0 A_k(x-\tfrac{\hbar}{2}y) d s.
\label{eq:PropositionAux2}
\end{align}
Now combine the first terms on the right-hand sides of \eqref{eq:PropositionAux1} and \eqref{eq:PropositionAux2}:
\begin{align*}
\int_0^{1/2} \left(A_k(x+\tfrac{\hbar}{2}y) -A_k(x+s\hbar y) \right)d s &= \int_0^{1/2}\left(\int_s^{1/2} \frac{d}{d\tau} A_k(x+\tau \hbar y) d\tau \right)d s \\
&= \int_0^{1/2}\left(\int_0^{\tau} \frac{d}{d\tau} A_k(x+\tau \hbar y) d s \right)d \tau \\
&= \int_0^{1/2} \frac{d}{d\tau} A_k(x+\tau \hbar y) \tau d \tau \\
&= \hbar y_{\ell} \int_0^{1/2} \partial_{\ell} A_k(x+\tau \hbar y) \tau d \tau,
\end{align*}
where we used Fubini's theorem in the second line. Similarly we combine the second terms on the right-hand sides of \eqref{eq:PropositionAux1} and \eqref{eq:PropositionAux2} and obtain 
\begin{align*}
\int_{-1/2}^{0}\left( A_k(x-\tfrac{\hbar}{2}y) -A_k(x+s\hbar y) \right)d s = \hbar y_{\ell} \int_{-1/2}^{0} \partial_{\ell} A_k(x+\tau \hbar y) \tau d \tau.
\end{align*}
Therefore, together with $\partial_kA_{\ell} -\partial_{\ell} A_k = \varepsilon_{jk\ell} B_j$,
\begin{align*}
    &\frac{1}{2}(A_k(X)+A_k(Y))+ \frac{1}{2}(\partial_{X_k}-\partial_{Y_k})\Gamma_A[X,Y] \\ &\qquad\qquad=  -\hbar y_\ell \int_{-1/2}^{1/2} \left(\partial_kA_\ell (x+\hbar sy) - \partial_{\ell} A_k(x+s \hbar y)\right) s d s,
    \\ &\qquad\qquad=  -\hbar y_\ell \int_{-1/2}^{1/2}  \varepsilon_{jk\ell} B_j(x+s\hbar y) s d s.
\end{align*}
Now \eqref{eq: cal D1j original variables} becomes in Weyl variables
\begin{equation}
\mathcal{D}_{1k} = i\partial_{y_k} - \hbar y_\ell \int_{-1/2}^{1/2}  \varepsilon_{jk\ell} B_j(x+s\hbar y) s d s.
\end{equation}
\underline{The operator $\mathcal{D}_{2j}$:}
Expressing \eqref{eq:Derivative(dX+dY)GammaA} in Weyl variables and switching to $s=\frac{1}{2}-t$,
\begin{align*}
(\partial_{X_k}+\partial_{Y_k})\Gamma_A[X,Y] = \widetilde{\Gamma_{\partial_k A}}[x,y]=&\int_0^1\partial_kA_\ell(x+\tfrac{\hbar}{2}y-\hbar t y)(-\hbar y_\ell)dt \\
=& -\hbar y_{\ell}\int_{-1/2}^{1/2} \partial_kA_\ell(x+\hbar s y)ds.
\end{align*}
Moreover,
\begin{align*}
   A_k(X)-A_k(Y) = -\int_0^1 \partial_{\ell}A_k(X+t(Y-X))(Y_{\ell}-X_{\ell}) d t.
\end{align*}
In Weyl variables,
\begin{align*}
   A_k(x+\tfrac{\hbar}{2}y)-A_k(x-\tfrac{\hbar}{2}y) = \hbar y_\ell \int_{-1/2}^{1/2} \partial_{\ell}A_k(x+\hbar s y) d s.
\end{align*}
Therefore, \eqref{eq: cal D2j original variables} becomes
\begin{equation}
\mathcal{D}_{2k} = \partial_{x_k} -i y_\ell \int_{-1/2}^{1/2} \left(\partial_{\ell}A_k(x+\hbar s y) - \partial_{k}A_{\ell}(x+\hbar s y) \right)d s,
\end{equation}
and again with $\partial_kA_{\ell} -\partial_{\ell} A_k = \varepsilon_{jk\ell} B_j$ we obtain
\begin{equation}
\mathcal{D}_{2k} = \partial_{x_k} + i y_\ell \int_{-1/2}^{1/2} \varepsilon_{jk\ell} B_j(x+\hbar s y) d s.
\end{equation}
This concludes the calculations of $\mathcal{D}_{1}$ and  $\mathcal{D}_{2}$ in Weyl variables.\\

\underline{The electric term:} 
First we note that
\begin{align*}
    \tfrac{1}{i\hbar}(\phi(x+\tfrac{\hbar y}{2})-\phi(x-\tfrac{\hbar y}{2})) &= \tfrac{1}{i\hbar}\int_{-1/2}^{1/2} \frac{d}{d\sigma}\phi(x+\sigma\hbar y)d\sigma  \\
    &= iy\cdot \int_{-1/2}^{1/2} (-\nabla_x \phi)(x+\sigma\hbar y)d\sigma
    \\
    &= iy\cdot \mathcal{I}_2[-\nabla_x \phi]. 
\end{align*}
Finally, we observe for the term involving the time derivative of $A$:
\begin{align*}
    \tfrac{i}{\hbar}\partial_t\widetilde{\Gamma_A}[x,y] &= \tfrac{i}{\hbar} \int_0^1 \partial_tA_k(x+\tfrac{\hbar}{2}y-\hbar t y)(-\hbar y_k) d t \\
    &= i y \cdot \int_{-1/2}^{1/2}(-\partial_tA)(x+\sigma\hbar  y) d \sigma \\
    &= i y \cdot \mathcal{I}_2[-\partial_t A].
\end{align*}
Combining both yields
\begin{align*}
    \left[\tfrac{1}{i\hbar}(\phi(x+\tfrac{\hbar y}{2})-\phi(x-\tfrac{\hbar y}{2})) + \frac{i}{\hbar}\partial_t \widetilde{\Gamma_A}[x,y]\right]\tilde{S}(x,y) = iy\cdot \mathcal{I}_2[E]\tilde{S}(x,y).
\end{align*}
This shows that $\tilde{S}$ obeys \eqref{eq:magnetic von Neumann}.

\underline{The magnetic Wigner equation:} For the magnetic Wigner equation we observe that $\mathcal{D}_{1k} \mathcal{D}_{2k} \tilde{S}(x,y)$ can be written as
\begin{align*}
 \mathcal{D}_{1k} \mathcal{D}_{2k} \tilde{S}(x,y) 
    &=\left(i\partial_{y_k} - \hbar (y \times \mathcal{I}_1[B])_k\right)\left(\partial_{x_k} + i (y \times  \mathcal{I}_2[B])_k\right) \tilde{S}(x,y) \\
    &= i\partial_{y_k}\partial_{x_k}\tilde{S}(x,y) + i\partial_{y_k} \left((iy \times  \mathcal{I}_2[B])_k\tilde{S}(x,y)\right) \\
    &\qquad -\hbar (y \times \mathcal{I}_1[B])_k\partial_{x_k}\tilde{S}(x,y) -i\hbar(y \times \mathcal{I}_1[B])_k(y \times  \mathcal{I}_2[B])_k \tilde{S}(x,y).
\end{align*}
Recall that $W^{\hbar}_A(x,v) = (2\pi)^{-d}\mathcal{F}_{y\rightarrow v}[\tilde{S}(x,\cdot)](v)$. Then, by taking the Fourier transform of \eqref{eq:magnetic von Neumann}, we obtain
\begin{align*}
    (2\pi)^{-d}\mathcal{F}_{y\rightarrow v}[i\partial_{y_k}\partial_{x_k}\tilde{S}(x,y)] &= -v_k \partial_{x_k}W^{\hbar}_A(x,v),
    \end{align*}
    and (recall the definition of $\Theta$ in \eqref{eq:def theta}) 
    \begin{align}
    (2\pi)^{-d}\mathcal{F}_{y\rightarrow v}\left[ - \partial_{y_k} \left((y \times  \mathcal{I}_2[B])_k\tilde{S}(x,y)\right)\right] &=v_k \varepsilon_{kj\ell} \partial_{v_j}(2\pi)^{-d}\mathcal{F}_{y\rightarrow v}\left[\mathcal{I}_2[B_{\ell}]\tilde{S}(x,y)\right] \nonumber \\
    &=\partial_{v_j} \varepsilon_{j\ell k}v_k(2\pi)^{-d}\mathcal{F}_{y\rightarrow v}\left[\mathcal{I}_2[B_{\ell}]\right] \ast_{v}W^{\hbar}_A(x,v) \nonumber \\
    &=-\partial_{v_j} \left(v\times \Theta[B] W^{\hbar}_A(x,v)\right)_j, \label{eq:WignerEquationLorentzForceB}
    \end{align}
    while (recall the definition of $\Omega$ in \eqref{eq:def theta})
    \begin{align}
     (2\pi)^{-d}\mathcal{F}_{y\rightarrow v}\left[-\hbar (y \times \mathcal{I}_1[B])_k\partial_{x_k}\tilde{S}(x,y)\right] &=-i\hbar \varepsilon_{kj\ell}\partial_{v_j}(2\pi)^{-d}\mathcal{F}_{y\rightarrow v}[\mathcal{I}_1[B_{\ell}]]\ast_v \partial_{x_k}W^{\hbar}_A(x,v) \nonumber\\
     &=i\hbar \partial_{v_j}\varepsilon_{jk \ell}(2\pi)^{-d}\mathcal{F}_{y\rightarrow v}[\mathcal{I}_1[B_{\ell}]]\ast_v \partial_{x_k}W^{\hbar}_A(x,v) \nonumber\\
     &=i\hbar \partial_{v_j}\varepsilon_{jk \ell}\partial_{x_k}\left((2\pi)^{-d}\mathcal{F}_{y\rightarrow v}[\mathcal{I}_1[B_{\ell}]]\ast_v W^{\hbar}_A(x,v)\right) \nonumber\\ &\qquad -i\hbar \partial_{v_j}(2\pi)^{-d}\mathcal{F}_{y\rightarrow v}[\varepsilon_{jk \ell}\partial_{x_k}\mathcal{I}_1[B_{\ell}]]\ast_v W^{\hbar}_A(x,v)
     \nonumber\\
     &=i\hbar \partial_{v_j}\left(\nabla_x \times\left(\Omega[B] W^{\hbar}_A(x,v)\right)\right)_j \nonumber \\ &\qquad -i\hbar \partial_{v_j}\Omega[(\nabla_x \times B)_{j}]W^{\hbar}_A(x,v).
     \label{eq:WignerEquationOhbarTerm1}
     \end{align}
     Moreover,
     \begin{align}
       (2\pi)^{-d}\mathcal{F}_{y\rightarrow v}\left[-i\hbar(y \times \mathcal{I}_1[B])_k(y \times  \mathcal{I}_2[B])_k \tilde{S}(x,y)\right]  &=i\hbar \epsilon_{jk\ell}\epsilon_{nkm} \partial_{v_{\ell}} \partial_{v_{m}} \Omega[B_j]\Theta[B_n]W_A^{\hbar}(x,v),\label{eq:WignerEquationOhbarTerm2}
    \end{align}
       while for the electric part we have
    \begin{align}
       (2\pi)^{-d}\mathcal{F}_{y\rightarrow v}[iy \cdot \mathcal{I}_2[E](x,y)\tilde{S}(x,y)] &= -\nabla_v\cdot (\Theta[E]W^{\hbar}_A(x,v)).
       \label{eq:WignerEquationElectricPartialtA}
\end{align}
Combining \eqref{eq:WignerEquationLorentzForceB}-\eqref{eq:WignerEquationElectricPartialtA} yields $\Phi[E,B]W^{\hbar}_A$.
This completes the proof.
\end{proof}

\section*{Acknowledgement}

J.M. acknowledges financial support from the Austrian Science Fund (FWF) via the Schrödinger grant 10.55776/J4840.

\section*{AI declaration}

AI tools (Claude Opus, ChatGPT Astra) were used to check the mathematical correctness of the paper.

\bibliographystyle{abbrv}
\bibliography{references}

@book{brezisFunctionalAnalysisSobolev2011,
	address = {New York, NY},
	series = {Universitext},
	title = {Functional {Analysis}, {Sobolev} {Spaces} and {Partial} {Differential} {Equations}},
	copyright = {https://www.springernature.com/gp/researchers/text-and-data-mining},
	isbn = {978-0-387-70913-0 978-0-387-70914-7},
	url = {https://link.springer.com/10.1007/978-0-387-70914-7},
	doi = {10.1007/978-0-387-70914-7},
	language = {en},
	urldate = {2026-11-16},
	publisher = {Springer, New York},
	author = {Brezis, Haim},
	year = {2011},
}

@article{stratonovichGaugeInvariantAnalog1956,
	title = {A gauge invariant analog of the {Wigner} distribution},
	volume = {1},
	journal = {Sov. Phys. Dokl.},
	author = {Stratonovich, R. L.},
	year = {1956},
	pages = {414--418},
}

@book{bouchutKineticEquationsAsymptotic2000,
	series = {Series in {Applied} {Mathematics}},
	title = {Kinetic equations and asymptotic theory},
	url = {https://hal.science/hal-00538692},
	urldate = {2026-04-03},
	publisher = {Elsevier},
	author = {Bouchut, François and Golse, François and Pulvirenti, Mario},
	year = {2000},
}

@article{golseVelocityAveragingNot2027,
	title = {Velocity {Averaging} ({Or} {Not}...) for the {Wigner} {Kinetic} {Equation} in the {Semiclassical} {Regime}},
	volume = {25},
	number = {2},
	journal = {to appear in Commun. Math. Sci.},
	author = {Golse, François and Möller, Jakob},
	year = {2027},
	note = {arXiv:2512.01529},
}

@article{golseVelocityAveragingLemmas2026,
	title = {Velocity averaging lemmas: classical, quantum and semi-classical},
	volume = {364},
	issn = {1631-073X, 1778-3569},
	shorttitle = {Velocity averaging lemmas},
	url = {https://comptes-rendus.academie-sciences.fr/mathematique/articles/10.5802/crmath.853/},
	doi = {10.5802/crmath.853},
	language = {en},
	number = {G3},
	urldate = {2026-10-28},
	journal = {Comptes Rendus. Mathématique},
	author = {Golse, François and Mauser, Norbert J. and Möller, Jakob},
	month = oct,
	year = {2026},
	pages = {689--699},
}

@article{filbetApproximationVonNeumannEquation2025,
	title = {On the approximation of the von-{Neumann} equation in the semi-classical limit. {Part} {I}: {Numerical} algorithm},
	volume = {527},
	issn = {00219991},
	shorttitle = {On the approximation of the von-{Neumann} equation in the semi-classical limit. {Part} {I}},
	url = {https://linkinghub.elsevier.com/retrieve/pii/S0021999125000932},
	doi = {10.1016/j.jcp.2025.113810},
	language = {en},
	urldate = {2026-08-03},
	journal = {Journal of Computational Physics},
	author = {Filbet, Francis and Golse, François},
	month = apr,
	year = {2025},
	pages = {113810},
}

@article{nedjalkovWignerEquationGeneral2019,
	title = {Wigner equation for general electromagnetic fields: {The} {Weyl}-{Stratonovich} transform},
	volume = {99},
	issn = {2469-9950, 2469-9969},
	shorttitle = {Wigner equation for general electromagnetic fields},
	url = {https://link.aps.org/doi/10.1103/PhysRevB.99.014423},
	doi = {10.1103/PhysRevB.99.014423},
	language = {en},
	number = {1},
	urldate = {2026-04-06},
	journal = {Physical Review B},
	author = {Nedjalkov, M. and Weinbub, J. and Ballicchia, M. and Selberherr, S. and Dimov, I. and Ferry, D. K.},
	month = jan,
	year = {2019},
	pages = {014423},
}

@phdthesis{lein_2011_thesis,
	address = {München},
	title = {Semiclassical {Dynamics} and {Magnetic} {Weyl} {Calculus}},
	doi = {arXiv:1202.4668v2},
	school = {Technische Universität},
	author = {Lein, Maximilian},
	year = {2011},
}

@incollection{ichinose_magnetic_2013,
	address = {Basel},
	title = {Magnetic {Relativistic} {Schrödinger} {Operators} and {Imaginary}-time {Path} {Integrals}},
	volume = {232},
	isbn = {978-3-0348-0590-2 978-3-0348-0591-9},
	url = {https://link.springer.com/10.1007/978-3-0348-0591-9_5},
	doi = {10.1007/978-3-0348-0591-9_5},
	language = {en},
	urldate = {2026-04-03},
	booktitle = {Mathematical {Physics}, {Spectral} {Theory} and {Stochastic} {Analysis}},
	publisher = {Springer Basel},
	author = {Ichinose, Takashi},
	editor = {Demuth, Michael and Kirsch, Werner},
	year = {2013},
	note = {Series Title: Operator Theory: Advances and Applications},
	pages = {247--297},
}

@article{perthameLimitingCaseVelocity1998,
	title = {A limiting case for velocity averaging},
	volume = {31},
	copyright = {http://www.elsevier.com/tdm/userlicense/1.0/},
	issn = {00129593},
	url = {http://linkinghub.elsevier.com/retrieve/pii/S0012959398801080},
	doi = {10.1016/S0012-9593(98)80108-0},
	language = {en},
	number = {4},
	urldate = {2026-04-03},
	journal = {Annales Scientifiques de l’École Normale Supérieure},
	author = {Perthame, B and Souganidis, P},
	month = jul,
	year = {1998},
	pages = {591--598},
}

@article{mullerProductRuleGauge1999,
	title = {Product rule for gauge invariant {Weyl} symbols and its application to the semiclassical description of guiding centre motion},
	volume = {32},
	issn = {0305-4470, 1361-6447},
	url = {https://iopscience.iop.org/article/10.1088/0305-4470/32/6/014},
	doi = {10.1088/0305-4470/32/6/014},
	number = {6},
	urldate = {2026-04-03},
	journal = {Journal of Physics A: Mathematical and General},
	author = {Müller, M},
	month = feb,
	year = {1999},
	pages = {1035--1052},
}

@article{serimaaGaugeindependentWignerFunctions1986,
	title = {Gauge-independent {Wigner} functions: {General} formulation},
	volume = {33},
	copyright = {http://link.aps.org/licenses/aps-default-license},
	issn = {0556-2791},
	shorttitle = {Gauge-independent {Wigner} functions},
	url = {https://link.aps.org/doi/10.1103/PhysRevA.33.2913},
	doi = {10.1103/PhysRevA.33.2913},
	language = {en},
	number = {5},
	urldate = {2026-04-03},
	journal = {Physical Review A},
	author = {Serimaa, O. T. and Javanainen, J. and Varró, S.},
	month = may,
	year = {1986},
	pages = {2913--2927},
}

@article{mantoiuMagneticWeylCalculus2004,
	title = {The magnetic {Weyl} calculus},
	volume = {45},
	issn = {0022-2488, 1089-7658},
	url = {https://pubs.aip.org/jmp/article/45/4/1394/383168/The-magnetic-Weyl-calculus},
	doi = {10.1063/1.1668334},
	language = {en},
	number = {4},
	urldate = {2026-03-08},
	journal = {Journal of Mathematical Physics},
	author = {Măntoiu, Marius and Purice, Radu},
	month = apr,
	year = {2004},
	pages = {1394--1417},
}

@article{dipernaGlobalWeakSolutions1989,
	title = {Global weak solutions of {Vlasov}-{Maxwell} systems},
	volume = {42},
	number = {6},
	journal = {Comm. Pure Appl. Math.},
	author = {DiPerna, Ronald J and Lions, Pierre-Louis},
	year = {1989},
	pages = {729--757},
}

@article{lionsMesuresWigner1993,
	title = {Sur les mesures de {Wigner}},
	volume = {9},
	number = {3},
	journal = {Rev. Mat. Iberoamericana},
	author = {Lions, Pierre-Louis and Paul, Thierry},
	year = {1993},
	pages = {553--618},
}

\end{document}